\documentclass{amsart}
\usepackage{amsfonts,amssymb,
amscd,amsmath,latexsym,amsbsy,bm}
\numberwithin{equation}{section}
\theoremstyle{plain}

\newtheorem{thm}{Theorem}[section]

\newtheorem{prop}[thm]{Proposition}

\newtheorem{defi}[thm]{Definition}
\newtheorem{lem}[thm]{Lemma}
\newtheorem{cor}[thm]{Corollary}

\theoremstyle{remark}
\newtheorem{rema}[thm]{Remark}

\newcommand{\Z}{\mathbb{Z}}

\newcommand{\C}{\mathbb{C}}

\title[Boundary quantum KZ equations and Bethe vectors]
{Boundary quantum Knizhnik-Zamolodchikov equations
and Bethe vectors}
\author{Nicolai Reshetikhin, Jasper Stokman, Bart Vlaar}
\address{N.R.: Department of Mathematics, University of California, Berkeley,
CA 94720, USA \& KdV Institute for Mathematics, University of Amsterdam,
Science Park 904, 1098 XH Amsterdam, The Netherlands \& ITMO University, 197101, Kronverkskii ave. 49, Saint Petersburg, Russia.}
\email{reshetik@math.berkeley.edu}
\address{J.S.: KdV Institute for Mathematics, University of Amsterdam,
Science Park 904, 1098 XH Amsterdam, The Netherlands \& IMAPP,
Radboud University, Heyendaalseweg 135, 6525 AJ Nijmegen, The Netherlands.}
\email{j.v.stokman@uva.nl}
\address{B.V.: KdV Institute for Mathematics, University of Amsterdam,
Science Park 904, 1098 XH Amsterdam, The Netherlands.}
\email{b.h.m.vlaar@uva.nl}
\subjclass[2000]{}
\begin{document}
\begin{abstract}
Solutions to boundary quantum Knizhnik-Zamolodchikov equations are constructed as bilateral sums involving ''off-shell'' Bethe vectors in case the reflection matrix is diagonal and only the $2$-dimensional representation of $\mathcal{U}_q(\widehat{\mathfrak{sl}}_2)$ is involved. 
We also consider their rational and classical degenerations.
\end{abstract}
\maketitle
\setcounter{tocdepth}{1}

\section{Introduction}

Let us start with a short historical outline of events which led to the boundary quantum Knizhnik-Zamolodchikov (qKZ) equations.

\subsection{Integrable systems with reflecting boundary conditions}

The integrability of classical and quantum field theories in $1+1$-dimensional space time is sensitive to boundary conditions.
A model which is integrable for periodic boundary conditions may not be integrable for Dirichlet or other boundary conditions. Integrability of classical and quantum field theories is intrinsically related to the Yang-Baxter equation and related algebraic structures.
Integrable boundary conditions have been studied in the late '70s and early '80s (see references in \cite{Sk}).
Corresponding algebraic structures were outlined by Cherednik \cite{CQKZ2} where he introduced the reflection equation and constructed some of its solutions, and in Sklyanin's work \cite{Sk} where he constructed a family of integrable spin chains when the solutions to the reflection equation from \cite{CQKZ2} (the so-called $K$-matrices) are diagonal.
In the same paper \cite{Sk} Sklyanin constructed eigenvectors and eigenvalues of the corresponding transfer matrix.

Another important relevant development was Cardy's work on boundary conformal field theories \cite{Car} where he described boundary conditions in conformal field theories which would retain conformal symmetry.

After these initial successes the study of integrable boundary conditions continued and took new heights with many important developments happening in the 90s.

In the paper \cite{JKKMW} vertex operators for quantum affine $sl_2$ were used to construct correlation functions.
One of the results was the observation that matrix elements of vertex operators with respect to boundary states satisfy a boundary qKZ equation from \cite{CQKZ}.
As a result solutions corresponding to the level one representation of quantum affine $sl_2$ were described explicitly.
For the latest developments in this direction see \cite{W}. See also \cite{Ko} for an overview of various results on boundary qKZ equations, vertex operators and construction of solutions to the boundary qKZ equation using the realization of vertex operators by a Heisenberg algebra (bosonization).

Integrable boundary conditions for quantum field theories in the context of factorized scattering were studied in \cite{GZ} and \cite{CDRS}.
One of the important results of this is the relation between integrable quantum field theories with reflecting boundary conditions and boundary conformal field theories.
In a similar way the integrability of line defects was studied in \cite{DMS}.
A comprehensive analysis of form factors in integrable
field theories with reflecting boundary conditions was done in \cite{BPT}.

Another interesting development was the relation between combinatorial problems such as alternating sign matrices and enumeration of plane partitions.
It appears that polynomial solutions to the boundary qKZ equation play an important role there \cite{DFZJ}.
In the same paper all relevant polynomial solutions to the boundary qKZ equation were constructed.

\subsection{qKZ equations and their solutions}

The Knizhnik-Zamolodchikov (KZ) equations were discovered in \cite{KZ} as a system of differential equations which defines conformal blocks in the Wess-Zumino-Witten conformal field theory.
A representation theoretical interpretation of KZ equations
as differential equations for intertwiners between certain representation of affine Kac-Moody algebras was given in \cite{TK}.

Examples of quantum KZ equations first appeared in works of Smirnov \cite{Sm} as fundamental equations for form-factors in integrable quantum field theories.
In \cite{FR} these equations were derived from representation theory of quantum affine algebras.

Solutions to KZ equations can be found by various methods.
The representation theoretical interpretation of KZ equations gives a natural construction of solutions as matrix elements of vertex operators. 
These matrix elements can be computed using various convenient realizations of highest weights representations of corresponding affine Lie algebras.
Wakimoto-type representations, also known as bosonic realizations, are particularly important.
These representations are infinite dimensional versions of the action of the Lie algebra on sections of a line bundle over the flag variety $G/B$, i.e. of the Borel-Weil-Bott constructions, see \cite{FB} for details.
Bosonic realizations give integral formulae for solutions to Knizhnik-Zamolodchikov equations.
For detailed analysis of these integral formulae and for references see \cite{SV}.
For details on how these integral formulae are related to representation theory
see \cite{EFK} and references therein.

These integral realizations involve ''off-shell'' Bethe vectors for Gaudin integrable systems \cite{B,RV}.
Similar construction of solutions to the qKZ equations involving Bethe vectors for the corresponding spin chains
and Jackson integrals was found in \cite{R} for quantum affine $sl_2$. It was immediately generalized to quantum affine $sl_n$ in \cite{TV}.
It is natural to expect that similar constructions exist for all simple Lie algebras and that they are particularly simple for classical Lie algebras for which Bethe vectors were constructed in \cite{R1,R2}.
However it seems that this has not been done yet. 
In this paper we construct solutions of the boundary qKZ equations using off-shell Bethe vectors \cite{Sk} for the corresponding spin chain with reflecting boundary conditions for diagonal boundary matrices and for two-dimensional representations of quantum $sl_2$. 
We discuss generalizations in the Conclusion.

\subsection{Cherednik's boundary qKZ equations}

Root system generalizations of the classical and quantum KZ equations were constructed by Cherednik \cite{ChUni,CQKZ}.
It depends on a so-called classical or quantum $R$-matrix datum, which can be associated to an arbitrary affine root system.
If the underlying affine root system is of type A then Cherednik's \cite{CQKZ} qKZ equations contain the Smirnov-Frenkel-Reshetikhin \cite{Sm,FR} qKZ equations as special cases. 
Boundary qKZ equations \cite[\S 5]{CQKZ} are Cherednik's qKZ equations associated to affine root systems of type B, C and D.

Representations of affine Hecke algebras produce quantum affine $R$-matrices \cite{CQKZ,ChAHA}. The double affine Hecke algebra plays an important role in this procedure. 
In case of principal series representations, the solutions of the associated qKZ equations are in one-to-one correspondence with suitable classes of common eigenfunctions of
Macdonald $q$-difference operators \cite{ChAHA,ChIND,StKZ}.
This correspondence explains the role of the Macdonald theory in the construction of polynomial solutions of qKZ equations in the context of the Razumov-Stroganov conjectures, cf., e.g., \cite{P,KT,Ka}.

Cherednik's theory relates to spin chains when the quantum affine $R$-matrix datum consists of linear operators on $W_1\otimes\cdots\otimes W_N$ for suitable vector spaces $W_i$ and the bulk $R$-matrices are given in terms of solutions $R^{ij}(x)\in\textup{End}(W_i\otimes W_j)$ ($i\not=j$) of the quantum Yang-Baxter equations
\begin{equation}\label{YBElarge}
R^{ij}(x) R^{ik}(x+y) R^{jk}(y) = R^{jk}(y)R^{ik}(x+y) R^{ij}(x)
\end{equation}
as linear operators on $W_i\otimes W_j\otimes W_k$.
The spin-$\frac{1}{2}$ XXZ chain corresponds to $W_i=\mathbb{C}^2$ with $R^{ij}(x)$ the $R$-matrix obtained from the two-dimensional evaluation representation of quantum $sl_2$.
This $R$-matrix also arises from a finite Hecke algebra (of type $A_{N-1}$) action on $\bigl(\mathbb{C}^2\bigr)^{\otimes N}$.

This action can be extended to an affine Hecke algebra action on $\bigl(\mathbb{C}^2\bigr)^{\otimes N}$ in various ways, thus leading to various extensions of the bulk $R$-matrices to a full quantum affine $R$-matrix datum.
For instance, it can be extended to an action of the affine Hecke algebra of type $C_N$, turning $\bigl(\mathbb{C}^2\bigr)^{\otimes N}$ into a principal series module.
In this case the associated quantum affine $R$-matrix datum is determined by the $R$-matrix and two {\it non-diagonal} solutions of the associated reflection equation ($K$-matrices). 
It provides a link between boundary qKZ equations associated to the spin-$\frac{1}{2}$ XXZ chain with non-diagonal reflecting boundary conditions and Koornwinder polynomials, see \cite{Ka,SV}.

In this paper we construct solutions of the boundary qKZ equations when the quantum $R$-matrix datum of affine type $C_N$ is constructed from the $R$-matrix of the spin-$\frac{1}{2}$ XXZ or XXX chain and the $K$-matrices are taken to be Cherednik's \cite{CQKZ2} diagonal solutions of the associated reflection equations.
As far as we know the corresponding quantum affine $R$-matrix datum does not arise from an action of the affine Hecke algebra of type $\textup{C}_N$ on $\bigl(\mathbb{C}^2\bigr)^{\otimes N}$.

\subsection{The boundary qKZ equations}

We formulate the boundary qKZ equations \cite[\S 5]{CQKZ} in case the quantum affine $R$-matrix datum arises from a single $R$-matrix $R(x)$ and two $K$-matrices $K^{\pm}(x)$.
We thus assume that $R(x)$ is a linear operator on $W\otimes W$ depending meromorphically on $x\in\mathbb{C}$ satisfying the quantum Yang-Baxter equation
\begin{equation}\label{YBE0}
R_{12}(x) R_{13}(x+y) R_{23}(y) = R_{23}(y)R_{13}(x+y) R_{12}(x)
\end{equation}
as linear operators on $W^{\otimes 3}$, where we use the well-known leg notation for linear operators acting on tensor product spaces.
The $K^{\pm}(x)$ are linear operators on $W$ depending meromorphically on $x\in\mathbb{C}$ and satisfying the reflection equations
\begin{equation}\label{Refl}
\begin{split}
R_{12}(x-y)K_1^{+}(x)R_{21}(x+y)K_2^{+}(y)&=K_2^{+}(y)R_{12}(x+y)K_1^{+}(x)R_{21}(x-y),\\
R_{21}(x-y)K_1^{-}(x)R_{12}(x+y)K_2^{-}(y)&=K_2^{-}(y)R_{21}(x+y)K_1^{-}(x)R_{12}(x-y)
\end{split}
\end{equation}
as linear operators on $W\otimes W$. We define for $r=1,\ldots,N$ transport matrices $A_r(\mathbf{t};\tau)$ as the linear operators
\begin{equation}\label{transportA}
\begin{split}
A_r(\mathbf{t};\tau):=&R_{r \, r\!+\!1}(t_r-t_{r+1}+\tau)\cdots R_{rN}(t_r-t_N+\tau)\\
\times &K^+_r(t_r+\tfrac{\tau}{2})
R_{Nr}(t_N+t_r)\cdots R_{r\!+\!1\, r}(t_{r+1}+t_r)\\
\times &R_{r\!-\!1 \,r}(t_{r-1}+t_r)\cdots
R_{1r}(t_1+t_r)K^-_r(t_r)\\
\times &R_{r1}(t_r-t_1)\cdots R_{r\,r\!-\!1}(t_r-t_{r-1})
\end{split}
\end{equation}
on $W^{\otimes N}$, depending meromorphically on $\mathbf{t}\in \mathbb{C}^{N}$.
The boundary qKZ equations are the following compatible system of difference equations
\begin{equation}\label{RqKZ}
f(\mathbf{t}+\tau \mathbf{e}_r)=A_r(\mathbf{t};\tau)f(\mathbf{t}),\qquad 1\leq r\leq N
\end{equation}
with step-size $\tau\in\C^\times$, where $f(\mathbf{t})$ is a
$W^{\otimes N}$-valued meromorphic function in
$\mathbf{t}=(t_1,\ldots,t_N)\in\C^N$ and $\mathbf{e}_r$ is the $r$-th standard basis vector in $\C^N$.

\subsection{The main result}
We found the following Jackson integral formula for solutions to the boundary qKZ equations \eqref{RqKZ} when $W=\mathbb{C}^2$, $R(x)$ is the $R$-matrix of the spin-$\frac{1}{2}$ XXZ or XXX chain, and the $K$-matrices $K^{\pm}(x)$ are taken to be Cherednik's \cite{CQKZ2} diagonal solutions $K^{\xi_{\pm}}(x)$ of the associated reflection equations \eqref{Refl}, which have an additional degree of freedom $\xi_{\pm}\in\mathbb{C}$ (see Section \ref{initialdata}).
This $R$-matrix $R(x)$ is $P$-symmetric, $R_{21}(x)=R(x)$, hence the two reflection equations \eqref{Refl} are the same.

Let $\overline{\mathcal{B}}^{\xi_-}(x;\mathbf{t})$ be the suitably rescaled and normalized upper-right matrix element of the boundary quantum monodromy operator (see section \ref{R-Mon}) and assume that $g_{\xi_+,\xi_-}(x)$, $h(x)$ and $F(x)$ are meromorphic functions in $x\in\C$ satisfying the functional equations
\begin{equation*}
\begin{split}
g_{\xi_+,\xi_-}(x+\tau)&=
\frac{b(\xi_--x-\tfrac{\eta}{2})b(\xi_+-x-\tfrac{\tau}{2}-
\tfrac{\eta}{2})}{b(\xi_-+x+\tau-\tfrac{\eta}{2})
b(\xi_++x+\tfrac{\tau}{2}-\tfrac{\eta}{2})}g_{\xi_+,\xi_-}(x),\\
h(x+\tau)&=\frac{b(x+\tau)b(x+\eta)}
{b(x)b(x+\tau-\eta)}h(x),\\
F(x+\tau)&=\frac{b(x+\tau-\tfrac{\eta}{2})}
{b(x+\tau+\tfrac{\eta}{2})}F(x),
\end{split}
\end{equation*}
where either $b(x)=\lambda\sinh(\nu x)$ (XXZ spin chain case) or $b(x)=\lambda x$ (XXX spin chain case) for some $\lambda,\nu\in\mathbb{C}^\times$.
Let $(\mathbf{e}_+,\mathbf{e}_-)$ be the standard basis of $\C^2$.
Fix generic $\mathbf{x}_0\in\C^M$ and suppose that the $\bigl(\mathbb{C}^2\bigr)^{\otimes N}$-valued sum
\begin{equation*}
\begin{split}
f_M(\mathbf{t}):=&\sum_{\mathbf{x}\in\mathbf{x}_0+\tau \Z^M}
\Bigl(\prod_{i=1}^M g_{\xi_+,\xi_-}(x_i)\Bigr)
\Bigl(\prod_{1\leq i<j\leq M}h(x_i+ x_j)h(x_i- x_j)\Bigr)\\
&\qquad\qquad\qquad\times\Bigl(\prod_{r=1}^N \prod_{i=1}^M F(t_r+ x_i)F(t_r- x_i)\Bigr)
\Bigl(\prod_{i=1}^M\overline{\mathcal{B}}^{\xi_-}(x_i;\mathbf{t})\Bigr) {\mathbf e}_+^{\otimes N}
\end{split}
\end{equation*}
converges mero-uniformly in $\mathbf{t}\in\C^N$.
Then $f_M$ is a meromorphic solution of the boundary qKZ equations \eqref{RqKZ}.

We will describe the functions $g_{\xi_+,\xi_-}(x)$, $h(x)$ and $F(x)$ explicitly (up to a doubly periodic function) in terms
of $q$-Gamma functions (XXZ spin chain case)
and Gamma functions (XXX spin chain case) respectively.
We will also give an explicit parameter domain for which the bilateral series mero-uniformly converge.

\subsection{The outline of the paper} In Section \ref{initialdata} we recall the definition of the XXX and XXZ spin-$\tfrac{1}{2}$ chain $R$-matrices, fix uniform notations and recall the definition of the diagonal $K$-matrices and the basic operations of the $R$- and $K$-matrices.
In the same section we recall the definition of monodromy operators and boundary monodromy operators and Sklyanin's \cite{Sk} construction of the Bethe vectors which produce the eigenvalues and eigenvectors of the boundary transfer operators. In Section \ref{qKZ-sol} we formulate the boundary qKZ equations and describe the main results in detail. We also discuss in Section \ref{qKZ-sol} the simple special case that the $R$-matrix is the identity operator, in which case the solutions are related to well-known bilateral sum evaluation formulas (Bailey's \cite{Ba}
${}_6\psi_6$ summation formula for the XXZ spin chain and Dougall's \cite{Do} ${}_5H_5$ summation formula for the XXX spin chain).
In Section \ref{Bsection} algebraic properties of the boundary monodromy operators, essential for the proof of the main result, are derived.
The proof of the main theorem is given in Section \ref{proofSection}.
In the conclusion we outline some open problems and the work in progress. In Appendix \ref{solKZ} we analyse boundary KZ equations, obtained as the formal $\tau\rightarrow 0$ limit of the boundary qKZ equations, and derive explicit integral solutions.

\subsection{Acknowledgements} 
We are grateful to J.-S. Caux,  P. Di Francesco, R. Kedem, B. Nienhuis, E. Opdam, M. Schlosser, F. Smirnov and P. Zinn-Justin for stimulating discussions. 
The work of N.R. was supported by the NSF grant DMS-0901431 and by the Chern-Simons research grant. 
He is grateful for the hospitality at QGM, Aarhus University, during December 2012 and January 2013 when a large part of this paper was completed.
The work of B.V. was supported by a NWO free competition grant. 
J.S. and B.V. thank the University of California for hospitality.

\section{$R$- and $K$-matrices and the boundary monodromy operator} \label{initialdata}
\subsection{$R$- and $K$-matrices}

Let $(\mathbf{e}_+,\mathbf{e}_-)$ be a fixed ordered basis of $\C^2$.
We represent linear operators on $\C^2$ by $2\times 2$-matrices with respect to $(\mathbf{e}_+,\mathbf{e}_-)$, and linear operators on $\C^2\otimes \C^2$ by $4\times 4$-matrices with respect to the ordered basis
$(\mathbf{e}_+\otimes \mathbf{e}_+,\mathbf{e}_+\otimes \mathbf{e}_-,\mathbf{e}_-\otimes \mathbf{e}_+,\mathbf{e}_-\otimes \mathbf{e}_-)$ of $\C^2\otimes\C^2$.
For a given nonzero meromorphic function $b$ in one variable set
\begin{equation}\label{R}
R(x;\eta):=
\frac{1}{b(x+\eta)} \begin{pmatrix}
b(x+\eta) & 0 & 0 & 0\\
0 & b(x) & b(\eta) & 0\\
0 & b(\eta) & b(x) & 0\\
0 & 0 & 0 & b(x+\eta) \end{pmatrix}.
\end{equation}
We view $R(x;\eta)$ as a linear operator on $\C^2\otimes\C^2$ depending meromorphically on the variables $x$ and $\eta$.
We use the standard tensor leg notation for linear operators on tensor product spaces.
\begin{lem}\label{Rmatrixlemma}
Let $b$ be a nonzero meromorphic function in one variable.
The following two statements are equivalent:
\begin{enumerate}
\item[{\bf a.}] $b$ satisfies the functional equation
\begin{equation}\label{fundamentalb} b(x)b(\eta)+b(x+y+\eta)b(y)=b(y+\eta)b(x+y)\end{equation}
as meromorphic functions in the variables $x,y,\eta$.
\item[{\bf b.}] $R(x;\eta)$ satisfies the quantum Yang-Baxter equation
\begin{equation}\label{YBE}
R_{12}(x;\eta)R_{13}(x+y;\eta)R_{23}(y;\eta)=
R_{23}(y;\eta)R_{13}(x+y;\eta)R_{12}(x;\eta)
\end{equation}
as linear operators on $\C^2\otimes\C^2\otimes\C^2$
depending meromorphically on $x,y,\eta$, with the tensor legs labelled by $1,2,3$ from left to right.
\end{enumerate}
\end{lem}
We fix from now on a nonzero meromorphic function $b$ in one variable satisfying the functional equation \eqref{fundamentalb}.
Then either $b(x)=\lambda x$ for some $\lambda\in\mathbb{C}^\times$ or $b(x)=\lambda\sinh(\nu x)$  for some $\lambda,\nu\in\mathbb{C}^\times$ (cf., e.g., \cite{BFV}). 
The associated $R$-matrix $R(x;\eta)$ is, up to normalization, either the $R$-matrix associated to the XXX  spin-$\tfrac{1}{2}$ chain or the $R$-matrix associated to the XXZ spin-$\tfrac{1}{2}$ chain, respectively.

The solution $R(x;\eta)$ of the quantum Yang-Baxter equation has the following properties.
\begin{itemize}
\item Regularity: $R(0;\eta)=P$, where $P$ is the permutation operator on $\C^2 \otimes \C^2$
defined by $P(\mathbf v \otimes \mathbf w) = \mathbf w \otimes \mathbf v$ ($\mathbf v, \mathbf w \in \C^2$).
\item Unitarity: $R(x;\eta)R(-x;\eta)=\textup{id}$; this motivates our choice of normalization.
\item $P$-symmetry: $R_{21}(x;\eta):=PR_{12}(x;\eta)P =R_{12}(x;\eta)$.
\item $T$-symmetry: $R(x;\eta)^T=R(x;\eta)$.
\item Crossing symmetry:
\begin{equation} \label{Rcrossing}
R^{T_1}_{12}(-x;\eta)=
\frac{b(x)}{b(x-\eta)} \sigma^y_1 R_{12}(x-\eta;\eta) \sigma^y_1,
\end{equation}
where $\sigma^y = \begin{pmatrix} 0 & -\sqrt{-1} \\ \sqrt{-1} & 0 \end{pmatrix}$ and $T_r$ indicates the (partial) transpose in the tensor leg labelled by $r$.
\end{itemize}

For the given $R$-matrix \eqref{R},
a $K$-matrix is a $2\times 2$-matrix $K(x)$ depending meromorphically on
$x\in\mathbb{C}$ which satisfies the reflection equation
\begin{equation}\label{reaux}
\begin{gathered}
R_{12}(x-y;\eta) K_1(x) R_{12}(x+y;\eta) K_2(y) = \qquad \\
\qquad = K_2(y) R_{12}(x+y;\eta) K_1(x) R_{12}(x-y;\eta)
\end{gathered}
\end{equation}
as operator-valued meromorphic functions in the variables $x,y,\eta$.
In this paper we focus on diagonal nonsingular $K$-matrices. Upon
normalization we may and will restrict to $K(x)$ of the form
\begin{equation}\label{Kdiag}
K(x) = \begin{pmatrix} 1 & 0 \\ 0 & \alpha(x) \end{pmatrix}
\end{equation}
for a suitable nonzero meromorphic function $\alpha(x)$ in $x\in\mathbb{C}$.

It is easy to check that the $2\times 2$-matrix $K(x)$ satisfies the reflection equation \eqref{reaux} if and only if
\begin{equation}\label{coeffcond33}
b(x+y)(\alpha(x)-\alpha(y))=b(x-y)(\alpha(x)\alpha(y)-1)
\end{equation}
as meromorphic functions in $x$ and $y$.
Note that the condition \eqref{coeffcond33} is independent of $\eta$.
We readily obtain
\begin{lem}
Suppose $b$ is a nonzero meromorphic function satisfying \eqref{fundamentalb}.
Let $\alpha$ be a nonzero meromorphic function.
Then the following is a complete list of solutions of \eqref{coeffcond33}.
\begin{enumerate}
\item[{\bf a.}] $\alpha\equiv \pm 1$.
\item[{\bf b.}] $\alpha(x)=\frac{b(\xi-x)}{b(\xi+x)}$ for some $\xi\in\C$.
\item[{\bf c.}] $\alpha(x)=e^{\pm 2\nu x}$
if $b(x)=\lambda\sinh(\nu x)$ for some $\lambda,\mu\in\C^\times$.
\end{enumerate}
\end{lem}
We take $\alpha(x)=\frac{b(\xi-x)}{b(\xi+x)}$ in the sequel (the other solutions are special cases or limit cases), and we treat $\xi$ as an additional free parameter in the theory.
The corresponding $K$-matrix
\begin{equation*}
K^\xi(x):=
\begin{pmatrix} 1 & 0\\
0 & \frac{b(\xi-x)}{b(\xi+x)}\end{pmatrix}
\end{equation*}
goes back to  \cite{CQKZ2}.

Note that the identity \eqref{coeffcond33} capturing the fact that $K^\xi(x)$
satisfies the reflection equation \eqref{reaux} becomes
\begin{equation}\label{coeffcond3alt}
\sum_{\epsilon_1,\epsilon_2\in\{\pm\}}
b(\epsilon_1y-\epsilon_2x)b(\xi+\epsilon_2y)b(\xi+\epsilon_1x)=0.
\end{equation}
Since $b(x)$ is an odd function,  \eqref{coeffcond3alt} is equivalent to
\begin{equation}\label{coeffcond3}
\sum_{\epsilon_1,\epsilon_2\in\{\pm\}}\epsilon_1\epsilon_2
\frac{b(\xi+\epsilon_1x)b(\xi+\epsilon_2y)}
{b(\epsilon_1x +\epsilon_2y)}=0.
\end{equation}

The diagonal $K$-matrix $K^\xi(x)$ satisfies the {\it boundary crossing symmetry}
\begin{equation}\label{Kcrossing}
\textup{Tr}_1 \Bigl(K^\xi_1(x+\eta) P_{12} R_{12}(2x) \Bigr)=
\frac{b(\xi+x)b(2(x+\eta))}{b(\xi+x+\eta)b(2x+\eta)} K^\xi_2(x),
\end{equation}
where $\textup{Tr}_1$ is the partial trace
over the first tensor component of $\C^2\otimes\C^2$. In fact, \eqref{Kcrossing} follows
from the identity
\begin{equation}\label{coeffcond4}
b(\xi+x)b(x-z)+b(\xi-x)b(x+z)=b(\xi-z)b(2x),
\end{equation}
which is a direct consequence of the fundamental identity \eqref{fundamentalb}.
The boundary crossing symmetry \eqref{Kcrossing}
coincides with the boundary reflection crossing in \cite[(3.35)]{GZ}.
Note furthermore that the conditions \eqref{coeffcond3alt} and \eqref{coeffcond3} capturing the reflection equation are consequences of \eqref{coeffcond4}.

\subsection{Monodromy operators} \label{R-Mon}

Let $S_N$ be the symmetric group in $N$ letters.
Set $V:=\bigl(\C^2\bigr)^{\otimes N}$ and $\mathbf{t}:=(t_1,\ldots,t_N)$.
For $w \in S_N$ define the linear operator $T_w(x;\mathbf{t})$ on $\C^2\otimes V$ by
\begin{equation*}
T_w(x;\mathbf{t}):= R_{0,w(1)}(x-t_{w(1)})
\cdots R_{0,w(N)}(x-t_{w(N)})=
\begin{pmatrix} A_w(x;\mathbf{t}) & B_w(x;\mathbf{t})\\
C_w(x;\mathbf{t}) & D_w(x;\mathbf{t})
\end{pmatrix}
\end{equation*}
(from now on copies of the auxiliary space $\C^2$ will always be labelled by $0$'s).
The matrix coefficients are linear operators on $V$, obtained by
representing $T_w(x;\mathbf{t})$ as $2\times 2$-matrix
with respect to the ordered
basis $(\mathbf{e}_+,\mathbf{e}_-)$ of the auxiliary space $\C^2$.
Note that
\begin{equation}\label{eigenvalues}
A_w(x;\mathbf{t})\Omega=\Omega, \qquad D_w(x;\mathbf{t})\Omega= \Bigl(\prod_{r=1}^N\frac{b(x-t_r)}{b(x-t_r+\eta)}\Bigr)\Omega
\end{equation}
with the pseudo-vacuum vector
\begin{equation}\label{pseudovacuum}
\Omega:=\mathbf{e}_+^{\otimes N}\in V.
\end{equation}

The monodromy operator is
$T(x;\mathbf{t}):=T_e(x;\mathbf{t})$ with $e\in S_N$ the neutral element.
We will also use the notations
$A(x;\mathbf{t})=A_e(x;\mathbf{t}),\ldots,
D(x;\mathbf{t})=D_e(x;\mathbf{t})$ (the dependence on $N$ is implicitly
captured by the number of rapidities, $\mathbf{t}=(t_1,\ldots,t_N)$).
Note that for $N=1$,
\begin{equation}\label{B1}
\begin{split}
A(x;t)=\begin{pmatrix}
1 & 0\\ 0 & \frac{b(x-t)}{b(x-t+\eta)}\end{pmatrix},
\qquad &
B(x;t) = \begin{pmatrix} 0 & 0\\ \frac{b(\eta)}{b(x-t+\eta)} & 0 \end{pmatrix},\\
C(x;t) = \begin{pmatrix} 0 & \frac{b(\eta)}{b(x-t+\eta)}\\
0 & 0 \end{pmatrix},\qquad
& D(x;t) = \begin{pmatrix} \frac{b(x-t)}{b(x-t+\eta)} & 0\\ 0 & 1
\end{pmatrix}.
\end{split}
\end{equation}

The operators
$T_w(x;\mathbf{t})$ satisfy the fundamental commutation relations
\begin{equation}\label{RTT}
R_{00^\prime}(x-y)
T_{w,0}(x;\mathbf{t})T_{w,0^\prime}(y;\mathbf{t})=
T_{w,0^\prime}(y;\mathbf{t})T_{w,0}(x;\mathbf{t})R_{00^\prime}(x-y)
\end{equation}
as linear operators on $\C^2\otimes\C^2\otimes V$
(with the first and second tensor leg labelled by $0$ and $0^\prime$,
respectively). Here $T_{w,0}(x;\mathbf{t})$ is acting on the first and
third tensor leg and $T_{w,0^\prime}(x;\mathbf{t})$ on the second and third
tensor leg.

We set for $w\in S_N$,
\begin{equation*}
\mathcal{U}^\xi_{w}(x;\mathbf{t}):=
T_{w,0}(x;\mathbf{t})^{-1}K^\xi_0(x)^{-1}T_{w,0}(-x;\mathbf{t})
=\begin{pmatrix} \mathcal{A}^\xi_w(x;\mathbf{t}) &
\mathcal{B}^\xi_w(x;\mathbf{t})\\
\mathcal{C}^\xi_w(x;\mathbf{t}) & \mathcal{D}^\xi_w(x;\mathbf{t})
\end{pmatrix}
\end{equation*}
as linear operator on $\C^2\otimes V$. Then
$\mathcal{U}^\xi(x;\mathbf{t}):=\mathcal{U}^\xi_e(x;\mathbf{t})$
is the boundary monodromy operator \cite{Sk} associated to the $K$-matrix
$K^\xi$. The operators $\mathcal{U}^\xi_w(x;\mathbf{t})$
satisfy the fundamental commutation relations
\begin{equation}\label{RU}
\begin{split}
R_{00^\prime}(y-x)\mathcal{U}^\xi_{w,0}(x;\mathbf{t})&R_{00^\prime}(-x-y)
\mathcal{U}^\xi_{w,0^\prime}(y;\mathbf{t})=\\
=&\mathcal{U}^\xi_{w,0^\prime}(y;\mathbf{t})R_{00^\prime}(-x-y)
\mathcal{U}^\xi_{w,0}(x;\mathbf{t})R_{00^\prime}(y-x)
\end{split}
\end{equation}
as linear operators on $\C^2\otimes\C^2\otimes V$.

\section{The boundary quantum KZ equations and its solutions}\label{qKZ-sol}
\subsection{The boundary quantum KZ equations}

Recall from the introduction that Cherednik's \cite{CQKZ} boundary quantum KZ equations with step-size $\tau\in\mathbb{C}^\times$ associated to
our particular choice $R(x)=R(x;\eta)$, $K^{\pm}(x)=K^{\xi_{\pm}}(x)$ of initial data ($\xi_{\pm}\in\mathbb{C}$) are defined as
\begin{equation}\label{ReflqKZ}
f(\mathbf{t}+\tau \mathbf{e}_r)=
A_r(\mathbf{t};\xi_+,\xi_-;\tau)f(\mathbf{t}),\qquad r=1,\ldots,N
\end{equation}
for $V$-valued meromorphic functions $f(\mathbf{t})$ in
$\mathbf{t}\in\C^N$, with the transport operator $A_r(\mathbf{t};\xi_+,\xi_-;\tau)$ the linear operator
on $V=\bigl(\C^2\bigr)^{\otimes N}$
defined by
\begin{equation}\label{Atauj}
\begin{split}
A_r(\mathbf{t};\xi_+,\xi_-;\tau)&=
R_{r \, r\!+\!1}(t_r-t_{r+1}+\tau)\cdots R_{rN}(t_r-t_N+\tau)\\
&\times K^{\xi_+}_r(t_r+\tfrac{\tau}{2})
R_{Nr}(t_N+t_r)\cdots R_{r\!+\!1 \, r}(t_{r+1}+t_r)\\
&\times R_{r\!-\!1 \, r}(t_{r-1}+t_r)\cdots
R_{1r}(t_1+t_r)K^{\xi_-}_r(t_r)\\
&\times R_{r1}(t_r-t_1)\cdots R_{r\, r\!-\!1}(t_r-t_{r-1}).
\end{split}
\end{equation}
The compatibility of the difference equations \eqref{ReflqKZ} follows from the consistency conditions
\begin{equation}\label{consistencyA}
A_r(\mathbf{t}+\tau \mathbf{e}_s;\xi_+,\xi_-;\tau) A_s(\mathbf{t};\xi_+,\xi_-;\tau)=
A_s(\mathbf{t}+\tau \mathbf{e}_r;\xi_+,\xi_-;\tau)
A_r(\mathbf{t};\xi_+,\xi_-;\tau)
\end{equation}
of the transport operators ($r,s=1,\ldots,N$).
Note that the pseudo-vacuum vector $\Omega$ is a constant solution
of the boundary quantum KZ equations \eqref{ReflqKZ}.

\subsection{The formulation of the main result}

Set for $w\in S_N$,
\[ 
\overline{\mathcal{B}}^\xi_w(x;\mathbf{t}):= \Bigl(\prod_{r=1}^N\frac{b(x-t_r-\tfrac{\eta}{2})}{b(x-t_r+\tfrac{\eta}{2})} \Bigr) \frac{b(\xi-x-\tfrac{\eta}{2})b(2x)}{b(2x+\eta)}\mathcal{B}^\xi_w(x+\tfrac{\eta}{2};\mathbf{t}).
\]
By the commutation relations \eqref{RU} for the boundary monodromy operator we have
$\lbrack \mathcal{B}^\xi_w(x;\mathbf{t}),
\mathcal{B}^\xi_w(y;\mathbf{t})\rbrack=0$, hence also
\[
\lbrack \overline{\mathcal{B}}^\xi_w(x;\mathbf{t}),
\overline{\mathcal{B}}^\xi_w(y;\mathbf{t})\rbrack=0.
\]

Set for $M \in \Z_{\geq 1}$ and $\mathbf{x}=(x_1,\ldots,x_M)$,
\[
\overline{\mathcal{B}}^{\xi,(M)}_w(\mathbf{x};\mathbf{t}):=\prod_{i=1}^M
\overline{\mathcal{B}}^\xi_w(x_i;\mathbf{t})
\]
and write $\overline{\mathcal{B}}^\xi(x;\mathbf{t})=
\overline{\mathcal{B}}^\xi_e(x;\mathbf{t})$ and
$\overline{\mathcal{B}}^{\xi,(M)}(\mathbf{x};\mathbf{t})=
\overline{\mathcal{B}}^{\xi,(M)}_e(\mathbf{x};\mathbf{t})$.

This renormalization of $\mathcal{B}^\xi(x;\mathbf{t})$ is motivated by the symmetries of the repeated actions of $\mathcal{B}^\xi(\cdot;\mathbf{t})$ on the pseudo-vacuum vector $\Omega$, to be discussed in detail in Section \ref{Bsection}.

We use the following notion of mero-uniformly convergent sums (cf. \cite{Ru}).
We formulate it here for scalar-valued functions; the extension to $V$-valued functions is obvious.
\begin{defi}
Let $\mathcal{C}\subset\C^M$ be a discrete subset and $w(\mathbf{x};\mathbf{t})$ ($\mathbf{x}\in\mathcal{C}$) a weight function with values depending meromorphically on $\mathbf{t}\in\C^N$.
Suppose that for all $\mathbf{t}_0\in\C^N$, there exists an open neighbourhood $U_{\mathbf{t}_0}\subset\C^N$ of $\mathbf{t}_0$ and a nonzero holomorphic function $v_{\mathbf{t}_0}$ on $U_{\mathbf{t}_0}$ such that
\begin{enumerate}
\item $v_{\mathbf{t}_0}(\mathbf{t})w(\mathbf{x};\mathbf{t})$ is holomorphic in $\mathbf{t}\in U_{\mathbf{t}_0}$ for all $\mathbf{x}\in\mathcal{C}$,
\item the sum $\sum_{\mathbf{x}\in\mathcal{C}}v_{\mathbf{t}_0}(\mathbf{t}) w(\mathbf{x};\mathbf{t})$ is absolutely and uniformly convergent for $\mathbf{t}\in U_{\mathbf{t}_0}$.
\end{enumerate}
Then there exists a unique meromorphic function $f(\mathbf{t})$ in $\mathbf{t}\in\C^N$ satisfying
\[ v_{\mathbf{t}_0}(\mathbf{t})f(\mathbf{t})=\sum_{\mathbf{x}\in\mathcal{C}} v_{\mathbf{t}_0}(\mathbf{t})w(\mathbf{x};\mathbf{t}) \]
for $\mathbf{t}\in U_{\mathbf{t}_0}$ and $\mathbf{t}_0\in\C^N$.
We will write
\[ f(\mathbf{t})=\sum_{\mathbf{x}\in\mathcal{C}}w(\mathbf{x};\mathbf{t}) \]
and we will say that the sum converges mero-uniformly.
\end{defi}
The following main result of the paper is the analogue of \cite[Thm. 1.4]{R} to the setup of the boundary quantum KZ equations.
Given a meromorphic function $h:\C \to \C$ we use the shorthand notation
 $h(x\pm y):=h(x+y)h(x-y)$.
\begin{thm}\label{mr}
Let $\xi_+,\xi_-\in\C$ and let $g_{\xi_+,\xi_-}$, $h$
and $F$ be meromorphic functions in one variable satisfying
the functional equations
\begin{equation*}
\begin{split}
g_{\xi_+,\xi_-}(x+\tau)&=
\frac{b(\xi_--x-\tfrac{\eta}{2})b(\xi_+-x-\tfrac{\tau}{2}-
\tfrac{\eta}{2})}{b(\xi_-+x+\tau-\tfrac{\eta}{2})
b(\xi_++x+\tfrac{\tau}{2}-\tfrac{\eta}{2})}g_{\xi_+,\xi_-}(x),\\
h(x+\tau)&=\frac{b(x+\tau)b(x+\eta)}
{b(x)b(x+\tau-\eta)}h(x),\\
F(x+\tau)&=\frac{b(x+\tau-\tfrac{\eta}{2})}
{b(x+\tau+\tfrac{\eta}{2})}F(x).
\end{split}
\end{equation*}
Fix $\mathbf{x}_0\in\C^M$ and suppose that the $V$-valued sum
\begin{equation*}
\begin{split}
f_M(\mathbf{t}):=\sum_{\mathbf{x} \in \mathbf{x}_0 + \tau \Z^M}
\Bigl(\prod_{i=1}^M g_{\xi_+,\xi_-}(x_i)\Bigr)
&\Bigl(\prod_{1\leq i<j\leq M}h(x_i\pm x_j)\Bigr)\\
&\quad\times\Bigl(\prod_{r=1}^N\prod_{i=1}^MF(t_r\pm x_i)\Bigr)
\overline{\mathcal{B}}^{\xi_-,(M)}(\mathbf{x};\mathbf{t})\Omega
\end{split}
\end{equation*}
converges mero-uniformly in $\mathbf{t}\in\C^N$.
Then $f_M$ is a meromorphic solution of the boundary quantum KZ equations \eqref{ReflqKZ}, i.e.
\[
f_M(\mathbf{t}+\tau \mathbf{e}_r)=
A_r(\mathbf{t};\xi_+,\xi_-;\tau)f_M(\mathbf{t}),\qquad r=1,\ldots,N.
\]
\end{thm}
The proof of Theorem \ref{mr} is given in Section \ref{proofSection}.

\subsection{The special case $M=1$ and $\eta=0$}
Since $R(x;0)=\textup{id}_{\C^2 \otimes \C^2}$ the boundary quantum KZ equations \eqref{ReflqKZ} reduce
for $\eta=0$ to the decoupled set of difference equations
\begin{equation}\label{rKZ0}
f^{\eta=0}(\mathbf{t}+\tau \mathbf{e}_r)=
K^{\xi_+}_r\left(t_r+\tfrac{\tau}{2}\right)K^{\xi_-}_r(t_r)
f^{\eta=0}(\mathbf{t}),\qquad r=1,\ldots,N.
\end{equation}
In view of Theorem \ref{mr} we take a meromorphic function
$g_{\xi_+,\xi_-}^{\eta=0}$
satisfying
\[
g_{\xi_+,\xi_-}^{\eta=0}(x+\tau)=
\frac{b(\xi_--x)b(\xi_+-x-\tfrac{\tau}{2})}
{b(\xi_-+x+\tau)
b(\xi_++x+\tfrac{\tau}{2})}g_{\xi_+,\xi_-}^{\eta=0}(x)
\]
(note that we can take $h\equiv 1\equiv F$ for $\eta=0$)
and we consider the resulting sum
\[
f_1^{\eta=0}(\mathbf{t}):=\sum_{x \in x_0 + \tau \Z}
g_{\xi_+,\xi_-}^{\eta=0}(x)\frac{1}{b(\eta)}
\overline{\mathcal{B}}^{\xi_-}(x;\mathbf{t})\Omega|_{\eta=0}.
\]
Provided mero-uniform convergence, Theorem \ref{mr} claims that it should satisfy the $\eta=0$ reduction \eqref{rKZ0} of the
boundary quantum KZ equations \eqref{ReflqKZ}. Since
\begin{equation}\label{eta0B}
\frac{1}{b(\eta)}\overline{\mathcal{B}}^{\xi}(x;\mathbf{t})\Omega|_{\eta=0}=
\sum_{r=1}^N\sum_{\epsilon\in\{\pm\}}\frac{\epsilon b(\xi-\epsilon x)}
{b(-\epsilon x-t_r)}\mathbf{e}_+^{\otimes (r-1)}\otimes \mathbf{e}_-\otimes \mathbf{e}_+^{\otimes (N-r)}
\end{equation}
by Corollary \ref{mathcalYcor}, $f_1^{\eta=0}(\mathbf{t})$ decouples as
\[
f_1^{\eta=0}(\mathbf{t})=\sum_{r=1}^NH(t_r;\xi_+,\xi_-)\mathbf{e}_+^{\otimes (r-1)}
\otimes \mathbf{e}_-\otimes \mathbf{e}_+^{\otimes (N-r)}
\]
with
\begin{equation}\label{Hdef}
H(t;\xi_+,\xi_-):=\sum_{x \in x_0 + \tau \Z} \sum_{\epsilon\in\{\pm\}}
\frac{\epsilon b(\xi_--\epsilon x)}{b(-\epsilon x-t)}
g_{\xi_+,\xi_-}^{\eta=0}(x).
\end{equation}
Hence $f_1^{\eta=0}(\mathbf{t})$ is a solution of the $\eta=0$ reduction \eqref{rKZ0} of the boundary
quantum KZ equations if and only if
\begin{equation}\label{Hfr}
H(t+\tau;\xi_+,\xi_-)=
\frac{b(\xi_--t)b(\xi_+-t-\tfrac{\tau}{2})}
{b(\xi_-+t)b(\xi_++t+\tfrac{\tau}{2})}
H(t;\xi_+,\xi_-)
\end{equation}
for the bilateral sum $H(t;\xi_+,\xi_-)$. We now
detail the proof of \eqref{Hfr},
which illustrates nicely the subtleties
also occurring in the general proof of Theorem \ref{mr}.

The affine Weyl group $W$ of type $\widehat{A}_1$ is the Coxeter group with generators $\gamma_0,\gamma_1$ and relations $\gamma_0^2=e=\gamma_1^2$.
Let $W$ act on $\C$ by $\gamma_0(t)=-t+\tau$ and $\gamma_1(t)=-t$.
Note that $(\gamma_0\gamma_1)(t)=t+\tau$.
Let $\{C_w^{\xi_+,\xi_-}\}_{w\in W}$ be the $W$-cocycle of meromorphic functions on $\C$ characterized by the properties ($u,v\in W$),
\begin{equation*}
\begin{split}
C_e^{\xi_+,\xi_-}(t)&=1,\qquad C_{uv}^{\xi_+,\xi_-}(t)=
C_u^{\xi_+,\xi_-}(t)C_v^{\xi_+,\xi_-}(u^{-1}t),\\
C_{\gamma_0}^{\xi_+,\xi_-}(t)&=
\frac{b(\xi_+-t+\tfrac{\tau}{2})}{b(\xi_++t-\tfrac{\tau}{2})},\qquad
C_{\gamma_1}^{\xi_+,\xi_-}(t)=\frac{b(\xi_-+t)}{b(\xi_--t)}.
\end{split}
\end{equation*}
The cocycle $\{C_w^{\xi_+,\xi_-}\}_{w\in W}$ gives rise to a twisted $W$-action on the space of meromorphic functions on $\C$ by
\[
\bigl(w\cdot f\bigr)(t):=C_w^{\xi_+,\xi_-}(t)f(w^{-1}t),\qquad w\in W.
\]
\begin{lem}\label{AWH}
Provided mero-uniform convergence, $H(\cdot;\xi_+,\xi_-)$ is
$W\cdot$-invariant.
\end{lem}
\begin{proof}
It suffices to prove the $\gamma_1\cdot$-and $\gamma_0\cdot$-invariance of $H(t;\xi_+,\xi_-)$, which are the respective statements
\begin{equation}\label{invarianceCoxeter}
\begin{split}
\frac{b(\xi_-+t)}{b(\xi_--t)}H(-t;\xi_+,\xi_-)&=H(t;\xi_+,\xi_-),\\
\frac{b(\xi_+-t+\tfrac{\tau}{2})}{b(\xi_++t-\tfrac{\tau}{2})}H(-t+\tau;\xi_+,
\xi_-)&=H(t;\xi_+,\xi_-).
\end{split}
\end{equation}

Observe that
\begin{equation}\label{application}
\sum_{\epsilon\in\{\pm\}}\frac{\epsilon b(\xi-\epsilon x)}{b(-\epsilon x-t)}=
\frac{b(\xi+t)b(2x)}{b(t\pm x)}
\end{equation}
due to the boundary crossing symmetry identity \eqref{coeffcond4}.
It follows that
\begin{equation}\label{Hcontracted}
\begin{split}
H(t;\xi_+,\xi_-)=b(\xi_-+t)\sum_{x \in x_0 + \tau \Z} \frac{b(2x)}{b(t\pm x)}
g_{\xi_+,\xi_-}^{\eta=0}(x),
\end{split}
\end{equation}
which immediately implies the $\gamma_1\cdot$-invariance of $H(t;\xi_+,\xi_-)$.

To prove the $\gamma_0\cdot$-invariance of $H(t;\xi_+,\xi_-)$ we write $H(t;\xi_+,\xi_-)=L_+(t)-L_-(t)$ with
\[
L_{\epsilon}(t):=\sum_{x \in x_0 + \tau \Z}
\frac{b(\xi_--\epsilon x)}{b(-\epsilon x-t)}g_{\xi_+,\xi_-}^{\eta=0}(x)
\]
and we replace in the expression of $L_+(t)$ the summation variable $x$ by $x-\tau$.
Using the functional equation satisfied by $g_{\xi_+,\xi_-}^{\eta=0}$ we obtain
\[
L_+(t)=\sum_{x\in x_0+\tau \Z}
\frac{b(\xi_-+x)b(\xi_++x-\tfrac{\tau}{2})}{b(-x-t+\tau)
b(\xi_+-x+\tfrac{\tau}{2})}
g_{\xi_+,\xi_-}^{\eta=0}(x)
\]
and consequently
\begin{equation*}
H(t;\xi_+,\xi_-)=
\sum_{x\in x_0+\tau \Z}\frac{b(\xi_-+x)Z(t;x)}{b(x+t-\tau)b(x-t)b(\xi_+-x+
\tfrac{\tau}{2})}g_{\xi_+,\xi_-}^{\eta=0}(x)
\end{equation*}
with
\[
Z(t;x):=b(t-x)b(\xi_++x-\tfrac{\tau}{2})+
b(\tau-x-t)b(\xi_+-x+\tfrac{\tau}{2}).
\]
By the boundary crossing symmetry identity \eqref{coeffcond4} we have
\[
Z(t;x)=
b(\xi_+-t+\tfrac{\tau}{2})b(\tau-2x),
\]
hence
\begin{equation*}
H(t;\xi_+,\xi_-)=
b(\xi_+-t+\tfrac{\tau}{2})\sum_{x\in x_0+\tau\Z}
\frac{b(\xi_-+x)b(\tau-2x)}{b(x+t-\tau)b(x-t)b(\xi_+-x+
\tfrac{\tau}{2})}g_{\xi_+,\xi_-}^{\eta=0}(x),
\end{equation*}
which immediately implies the $\gamma_0\cdot$-invariance of $H(t;\xi_+,\xi_-)$.
\end{proof}
Lemma \ref{AWH} implies
\[H(t+\tau;\xi_+,\xi_-)=H(\gamma_0\gamma_1t;\xi_+,\xi_-)=
C_{\gamma_1}^{\xi_+,\xi_-}(t)^{-1}C_{\gamma_0}^{\xi_+,\xi_-}(\gamma_1t)^{-1}
H(t;\xi_+,\xi_-),
\]
which is \eqref{Hfr}.
Note that in the proof of Lemma \ref{AWH} we only have used that
$b(x)$ is a nonzero meromorphic function in $x\in\C$
satisfying the boundary crossing symmetry
identity \eqref{coeffcond4}.

Note that the difference equation \eqref{Hfr} is equivalent to
\begin{equation}\label{special00}
\frac{b(\xi_-+t)}{b(\xi_--t)}H(-t;\xi_+,\xi_-)=
\frac{b(\xi_+-t+\tfrac{\tau}{2})}{b(\xi_++t-\tfrac{\tau}{2})}
H(-t+\tau;\xi_+,\xi_-),
\end{equation}
while Lemma \ref{AWH} is the stronger statement that both sides of
\eqref{special00} are equal to $H(t;\xi_+,\xi_-)$. Note that \eqref{special00} is equivalent to
\begin{equation}\label{special000}
\frac{b(\xi_+-t-\tfrac{\tau}{2})}{b(\xi_++t+\tfrac{\tau}{2})}
H(t;\xi_+,\xi_-)=
\frac{b(\xi_-+t)}{b(\xi_--t)}H(t+\tau;\xi_+,\xi_-).
\end{equation}
The general proof of Theorem \ref{mr} amounts to establishing a much more involved version of \eqref{special000}.

\begin{rema}
For arbitrary $M\geq 1$ it is easy to show that
\[
\frac{1}{b(\eta)^M}\overline{\mathcal{B}}_{\xi_-}^{(M)}(\mathbf{x};\mathbf{t})
|_{\eta=0}=
\sum_{\mathbf{m}}
\left(\sum_{w\in S_M}\prod_{r=1}^M\frac{b(\xi_-+t_{m_{r}})b(2x_{w(r)})}
{b(t_{m_{r}}\pm x_{w(r)})}\right)\sigma_{\mathbf{m}}^-
\]
with the summation over subsets $\mathbf{m}=\{m_i\}_{i=1}^M$ of $\{1,\ldots.N\}$ of cardinality $M$,
with $\sigma_{\mathbf{m}}^-:=\sigma_{m_1}^-\sigma_{m_2}^-\cdots\sigma_{m_M}^-$, and with $\sigma^-:=\left(\begin{matrix} 0 & 0\\ 1 & 0\end{matrix}\right)$.
It leads to the explicit expression
\[
f_M^{\eta=0}(\mathbf{t})=
\sum_{\mathbf{m}}\left(\sum_{w\in S_M}\prod_{r=1}^M
H_{x_{0,r}}(t_{m_{w(r)}};\xi_+,\xi_-)\right)\sigma_{\mathbf{m}}^-\Omega,
\]
from which it now directly follows that it satisfies the associated $\eta=0$ reduction
\[
f_M^{\eta=0}(\mathbf{t}+\tau\mathbf{e}_r)=K_{\xi_+,r}(t_r+\tfrac{\tau}{2})K_{\xi_-,r}(t_r)
f_M^{\eta=0}(\mathbf{t}),\qquad r=1,\ldots,N
\]
of the boundary quantum KZ equations \eqref{ReflqKZ}.
\end{rema}

\subsection{XXX spin-$\frac{1}{2}$ chain case}
In this subsection we make Theorem \ref{mr} concrete for
the initial data $b(x)=x$
related to the XXX spin-$\frac{1}{2}$ chain case.
By rescaling the summation variable
we may and will assume throughout this subsection that $\tau=-1$.

Solutions $g_{\xi_+,\xi_-}, h$ and $F$ of the resulting
functional relations
can be given in terms of Gamma functions,
\begin{equation*}
\begin{split}
g_{\xi_+,\xi_-}(x)&=\frac{\Gamma(x+\xi_--\tfrac{\eta}{2})
\Gamma(x+\xi_++\frac{1}{2}-\tfrac{\eta}{2})}
{\Gamma(x-\xi_-+1+\tfrac{\eta}{2})\Gamma(x-\xi_++\frac{1}{2}+\tfrac{\eta}{2})},\\
h(x)&=\frac{x\Gamma(x-\eta)}{\Gamma(x+1+\eta)},\\
F(x)&=\frac{\Gamma(x+\tfrac{\eta}{2})}{\Gamma(x-\tfrac{\eta}{2})}.
\end{split}
\end{equation*}
The resulting $V$-valued hypergeometric sum
\begin{equation*}
\begin{split}
f_M(\mathbf{t}):=\sum_{\mathbf{x}\in\mathbf{x}_0+\Z^M}
\Bigl(\prod_{i=1}^Mg_{\xi_+,\xi_-}(x_i)\Bigr)
&\Bigl(\prod_{1\leq i<j\leq M}h(x_i\pm x_j)\Bigr)\\
&\quad\times\Bigl(\prod_{r=1}^N\prod_{i=1}^MF(t_r\pm x_i)\Bigr)
\overline{\mathcal{B}}^{\xi_-,(M)}(\mathbf{x};\mathbf{t})\Omega
\end{split}
\end{equation*}
for generic $\mathbf{x}_0\in\C^M$
converges mero-uniformly in $\mathbf{t}\in
\C^N$ if $\Re(\eta)\geq 0$ and
\begin{equation}\label{convergence}
\Re\bigl(2\xi_++2\xi_-+2(N-1)\eta\bigr)<0.
\end{equation}
In fact, the boundary Bethe vector
$\overline{\mathcal{B}}^{\xi_-,(M)}(\mathbf{x};\mathbf{t}) \Omega$ is mero-uniformly bounded (with the obvious notion of mero-uniform boundedness).
This follows from \eqref{key} and the fact that $b(x)B(x;\mathbf{t})$
is mero-uniformly bounded
for $x\in x_0+\Z$ (generic $x_0$). By Stirling's asymptotic
formula for the Gamma function,
the factors $h(x_i\pm x_j)$ are bounded since $\Re(\eta)\geq 0$.
Applying Stirling's formula to the remaining factors of the summands leads to the convergency condition \eqref{convergence}.

Hence, under these assumptions on the parameters, the hypergeometric sum $f_M(\mathbf{t})$ is a meromorphic solution of the associated rational boundary quantum KZ equations \eqref{ReflqKZ} by Theorem \ref{mr}.

For $M=1$ and $\eta=0$ and the associated bilateral sum
$H(t;\xi_+,\xi_-)$ is
\begin{equation*}
\begin{split}
H(t;&\xi_+,\xi_-)=\\
&=-2\sum_{m=-\infty}^{\infty}
\frac{\Gamma(x_0+\xi_-+m)\Gamma(x_0+\xi_++\frac{1}{2}+m)}
{\Gamma(x_0-\xi_-+m+1)\Gamma(x_0-\xi_++\frac{1}{2}+m)}
\frac{(\xi_-+t)(x_0+m)}{(x_0+t+m)(x_0-t+m)}\\
&=-2\frac{(\xi_-+t)x_0}{(x_0+t)(x_0-t)}\frac{\Gamma(x_0+\xi_-)
\Gamma(x_0+\xi_++\frac{1}{2})}
{\Gamma(x_0-\xi_-+1)\Gamma(x_0-\xi_++\frac{1}{2})}\\
&\times{}_5H_5\left(\begin{matrix}
x_0+1 & x_0+t & x_0-t & x_0+\xi_- & x_0+\xi_++\frac{1}{2}\\
x_0 & x_0+t+1 & x_0-t+1 & x_0-\xi_-+1 & x_0-\xi_++\frac{1}{2}\end{matrix};1\right)
\end{split}
\end{equation*}
in view of \eqref{Hcontracted},
where we use the standard notation for bilateral hypergeometric
series. This bilateral sum is convergent for $\Re(2\xi_++2\xi_--1)<0$.
The slightly larger convergency domain compared to $f_1(\mathbf{t})$
comes from the fact that the boundary Bethe wave vector
at $\eta=0$ decays at infinity if $b(x)=x$,
see \eqref{eta0B} and \eqref{application}.

By Dougall's \cite{Do} ${}_5H_5$ summation formula
(see \cite[(2.3)]{Ba} for the summation formula in the present notations),
\[
H(t;\xi_+,\xi_-)=\textup{cst}
\frac{(\xi_-+t)\Gamma(x_0\pm t)\Gamma(1-x_0\pm t)}
{\Gamma(-\xi_-+1\pm t)\Gamma(-\xi_++\frac{1}{2}\pm t)}
\]
with $\textup{cst}$ independent of $t$, from which the difference equation
\eqref{Hfr} with $\tau=-1$ can be alternatively derived using
the functional equation $\Gamma(x+1)=x\Gamma(x)$ of the Gamma function.

\subsection{XXZ spin-$\frac{1}{2}$ chain case}
In this subsection we make Theorem \ref{mr} concrete for
the initial data $b(x)=\sinh(x)$
related to the XXZ spin-$\frac{1}{2}$ chain case.
We set $q:=e^\tau$ and we assume throughout this subsection that $\Re(\tau)<0$, so that
$|q|<1$.

Solutions $g_{\xi_+,\xi_-}, h$ and $F$ of the resulting
functional relations
can now be given in terms of $q$-Gamma functions or, equivalently, in terms
of $q$-shifted factorials
\[
\bigl(x;q\bigr)_{\infty}:=\prod_{i=0}^{\infty}(1-q^ix).
\]
We write $\bigl(x_1,\ldots,x_s;q\bigr)_{\infty}:=
\prod_{i=1}^s\bigl(x_i;q\bigr)_{\infty}$ for products of $q$-shifted
factorials. As solutions of the functional equations we take
\begin{equation*}
\begin{split}
g_{\xi_+,\xi_-}(x)&=e^{(\frac{2(\xi_-+\xi_+-\eta)}{\tau}+1)x}
\frac{\bigl(q^2e^{2(x+\xi_-)-\eta},qe^{2(x+\xi_+)-\eta};q^2\bigr)_{\infty}}
{\bigl(e^{2(x-\xi_-)+\eta},qe^{2(x-\xi_+)+\eta};q^2\bigr)_{\infty}},\\
h(x)&=e^{-\frac{2\eta x}{\tau}}(1-e^{2x})\frac{
\bigl(q^2e^{2(x-\eta)};q^2\bigr)_{\infty}}
{\bigl(e^{2(x+\eta)};q^2\bigr)_{\infty}},\\
F(x)&=e^{\frac{\eta x}{\tau}}\frac{\bigl(q^2e^{2x+\eta};q^2\bigr)_{\infty}}
{\bigl(q^2e^{2x-\eta};q^2\bigr)_{\infty}}.
\end{split}
\end{equation*}
The resulting $V$-valued basic hypergeometric sum
\begin{equation*}
\begin{split}
f_M(\mathbf{t}):=\sum_{\mathbf{x}\in\mathbf{x}_0+\tau\Z^M}
\Bigl(\prod_{i=1}^Mg_{\xi_+,\xi_-}(x_i)\Bigr)
&\Bigl(\prod_{1\leq i<j\leq M}h(x_i\pm x_j)\Bigr)\\
&\quad\times
\Bigl(\prod_{r=1}^N\prod_{i=1}^MF(t_r\pm x_i)\Bigr)
\overline{\mathcal{B}}^{\xi_-,(M)}(\mathbf{x};\mathbf{t})\Omega
\end{split}
\end{equation*}
converges mero-uniformly in $\mathbf{t}\in
\C^N$ for generic $\mathbf{x}_0\in\C^M$
if $\Re(\tau)<0$, $\Re(\eta)\geq 0$ and
\[
\Re(2\xi_++2\xi_-+2(N-1)\eta+\tau)<0.
\]
Under these assumptions, the basic hypergeometric sum $f_M(\mathbf{t})$ is a meromorphic solution of the associated trigonometric boundary quantum KZ equations \eqref{ReflqKZ} by Theorem \ref{mr}.

For $M=1$ and $\eta=0$ the associated bilateral sum
$H(t;\xi_+,\xi_-)$ (see \eqref{Hcontracted}) is
\begin{equation*}
\begin{split}
H(t;\xi_+,\xi_-)=
\sinh(\xi_-+t)&\sum_{x\in x_0+\tau\Z}\frac{(e^{4x}-1)}
{(1-e^{2(x+t)})(1-e^{2(x-t)})}\\
&\qquad\qquad\times
\frac{\bigl(q^2e^{2(x+\xi_-)},qe^{2(x+\xi_+)};q^2\bigr)_{\infty}}
{\bigl(e^{2(x-\xi_-)},qe^{2(x-\xi_+)};q^2\bigr)_{\infty}}
e^{(\frac{2(\xi_-+\xi_+)}{\tau}+1)x}.
\end{split}
\end{equation*}
It can be expressed as a very-well-poised ${}_6\psi_6$-series,
which in turn
can be evaluated using Bailey's \cite[(4.7)]{Ba} summation formula,
see \cite[Exerc. 5.17]{GR}. It gives
\[
H(t;\xi_+,\xi_-)=\textup{cst} (e^{\xi_-+t}-e^{-\xi_--t})
\frac{\bigl(q^2e^{2(\xi_-+t)},q^2e^{2(\xi_--t)},
qe^{2(\xi_++t)},qe^{2(\xi_+-t)};q^2\bigr)_{\infty}}
{\bigl(e^{2(x_0+t)},e^{2(x_0-t)},q^2e^{-2(x_0+t)},q^2e^{-2(x_0-t)};q^2\bigr)_{\infty}}
\]
with $\textup{cst}$ independent of $t$, from which the functional equation
\eqref{Hfr} can be alternatively derived using the functional equation $\bigl(x;q^2\bigr)_{\infty}=
(1-x)\bigl(q^2x;q^2\bigr)_{\infty}$.

\section{Properties of the boundary monodromy operators}\label{Bsection}
In this section we express the boundary $B$-operator $\overline{\mathcal{B}}^{\xi_-,(M)}(\mathbf{x};\mathbf{t})$ in terms of the operators $B(x;\mathbf{t})$ and $D(x;\mathbf{t})$ and we single out its explicit dependence on a fixed rapidity $t_r$.
Applied to the pseudo-vacuum vector $\Omega$, it gives an explicit expression of Sklyanin's \cite{Sk} boundary Bethe vector $\overline{\mathcal{B}}^{\xi_-,(M)}(\mathbf{x};\mathbf{t})\Omega$ in terms of the usual Bethe vectors $B^{(M)}(\mathbf{x};\mathbf{t})\Omega$ and it gives its explicit dependence on $t_r$.
These results will play a crucial role in the proof of Theorem \ref{mr} (see Section \ref{proofSection}).

\subsection{The $B$-operators}

By \eqref{RTT} we have for $w\in S_N$,
\[
\lbrack B_w(x;\mathbf{t}), B_w(y;\mathbf{t})\rbrack=0=
\lbrack D_w(x;\mathbf{t}), D_w(y;\mathbf{t})\rbrack,
\]
and for $\mathbf{y}=(y_1,\ldots,y_M)$,
\begin{equation}\label{Acommrelsigma}
\begin{split}
A_w(x;\mathbf{t})B_w^{(M)}(\mathbf{y};\mathbf{t})&=
\Bigl(\prod_{j=1}^M\frac{b(y_j-x+\eta)}{b(y_j-x)}\Bigr)
B_w^{(M)}(\mathbf{y};\mathbf{t})
A_w(x;\mathbf{t})\\
-\sum_{j=1}^M&\frac{b(\eta)}{b(y_j-x)}\Bigl(\prod_{i\not=j}
\frac{b(y_i-y_j+\eta)}{b(y_i-y_j)}\Bigr)
B_w^{(M)}((x,\mathbf{y}_{\widehat{\jmath}});\mathbf{t})A_w(y_j;\mathbf{t}),\\
D_w(x;\mathbf{t})B_w^{(M)}(\mathbf{y};\mathbf{t})&=
\Bigl(\prod_{j=1}^M\frac{b(x-y_j+\eta)}{b(x-y_j)}\Bigr)
B_w^{(M)}(\mathbf{y};\mathbf{t})
D_w(x;\mathbf{t})\\
-\sum_{j=1}^M&\frac{b(\eta)}{b(x-y_j)}\Bigl(\prod_{i\not=j}
\frac{b(y_j-y_i+\eta)}{b(y_j-y_i)}\Bigr)
B_w^{(M)}((x,\mathbf{y}_{\widehat{\jmath}});\mathbf{t})D_w(y_j;\mathbf{t}),
\end{split}
\end{equation}
where $\mathbf{y}_{\widehat{\jmath}}:=(y_1,\ldots,y_{j-1},y_{j+1},\ldots,y_M)$
and
\[
B_w^{(M)}(\mathbf{y};\mathbf{t}):=
B_{w}(y_1;\mathbf{t})B_w(y_2;\mathbf{t})\cdots
B_w(y_M;\mathbf{t})
\]
(we use the convention that empty products are equal to one).
The commutation relations \eqref{Acommrelsigma} are instrumental for the derivation of the Bethe ansatz equations of the Heisenberg $XXZ$ spin-$\frac{1}{2}$ chain with quasi-periodic boundary conditions using the algebraic Bethe ansatz, see, e.g., \cite{TF}.

\begin{prop}\label{straightening}
For $M\in\Z_{\geq 1}$ and $w\in S_N$,
\begin{equation}\label{straight}
\begin{split}
\overline{\mathcal{B}}^{\xi,(M)}_w(\mathbf{x};\mathbf{t})=
\sum_{\bm \epsilon}&\Bigl(\prod_{i=1}^M \epsilon_i b(\xi-\epsilon_ix_i-\tfrac{\eta}{2}) \Bigr)
\Bigl(\prod_{1\leq i<j\leq M} \frac{b(\epsilon_ix_i+\epsilon_jx_j+\eta)}
{b(\epsilon_ix_i+\epsilon_jx_j)}\Bigr)\\
&\qquad\qquad\quad\times
\Bigl(\prod_{i=1}^M B_w(-\epsilon_ix_i-\tfrac{\eta}{2};\mathbf{t})\Bigr)
\prod_{j=1}^M D_w(\epsilon_jx_j-\tfrac{\eta}{2};\mathbf{t})
\end{split}
\end{equation}
with the sum running over $\bm \epsilon=(\epsilon_1,\ldots,\epsilon_M)
\in\{\pm\}^M$.
\end{prop}
\begin{proof}
By the crossing symmetry of the $R$-matrix \eqref{Rcrossing}
we have
\[
\mathcal{U}^\xi_{w}(x;\mathbf{t})=\Bigl(
\prod_{r=1}^N\frac{b(x-t_r)}{b(x-t_r-\eta)}\Bigr)
\sigma^y_0T_{w,0}(x-\eta;\mathbf{t})^{T_0}
\sigma^y_0K^\xi_0(x)^{-1}T_{w,0}(-x;\mathbf{t}),
\]
hence
\begin{equation*}
\begin{split}
\overline{\mathcal{B}}^\xi_w(x;\mathbf{t})=\frac{b(2x)}{b(2x+\eta)}
&\Bigl(b(\xi-x-\tfrac{\eta}{2})D_w(x-\tfrac{\eta}{2};\mathbf{t})
B_w(-x-\tfrac{\eta}{2};\mathbf{t})\\
&\qquad\qquad -b(\xi+x+\tfrac{\eta}{2})B_w(x-\tfrac{\eta}{2};\mathbf{t})
D_w(-x-\tfrac{\eta}{2};\mathbf{t})\Bigr).
\end{split}
\end{equation*}
Commuting $D_w(y;\mathbf{t})$ and $B_w(x;\mathbf{t})$ using
\eqref{Acommrelsigma} gives
\begin{equation*}
\begin{split}
b(2x+\eta)&\overline{\mathcal{B}}^\xi_w(x;\mathbf{t})=
b(\xi-x-\tfrac{\eta}{2})b(2x+\eta)B_w(-x-\tfrac{\eta}{2};\mathbf{t})
D_w(x-\tfrac{\eta}{2};\mathbf{t})\\
&-\bigl(b(\xi+x+\tfrac{\eta}{2})b(2x)+
b(\xi-x-\tfrac{\eta}{2})b(\eta)\bigr)
B_w(x-\tfrac{\eta}{2};\mathbf{t})D_w(-x-\tfrac{\eta}{2};\mathbf{t}).
\end{split}
\end{equation*}
By the boundary crossing symmetry identity \eqref{coeffcond4}
this simplifies to
\[
\overline{\mathcal{B}}^\xi_w(x;\mathbf{t})=
\sum_{\epsilon\in\{\pm\}}\epsilon b(\xi-\epsilon x-\tfrac{\eta}{2})
B_w(-\epsilon x-\tfrac{\eta}{2};\mathbf{t})
D_w(\epsilon x-\tfrac{\eta}{2};
\mathbf{t}).
\]
Write $L_w^{(M)}$ for the right hand side of \eqref{straight}.
To prove the induction step we write for $M\geq 2$,
\[
\overline{\mathcal{B}}^{\xi,(M)}_w(\mathbf{x};\mathbf{t})=
\overline{\mathcal{B}}^\xi_w(x_1;\mathbf{t})
\overline{\mathcal{B}}^{\xi,(M-1)}_w(x_2,\ldots,x_M;\mathbf{t}),
\]
use the induction hypothesis, and commute
$B_w^{(M-1)}(\mathbf{y};\mathbf{t})$
and $D_w(x;\mathbf{t})$ using \eqref{Acommrelsigma}.
The resulting formula can be rewritten as
\begin{equation*}
\begin{split}
L_w^{(M)}-\overline{\mathcal{B}}^{\xi,(M)}_w(\mathbf{x};\mathbf{t})=
\sum_{i=2}^M\sum_{\bm \epsilon}k_i(\bm \epsilon)
&B_w(\pm x_1-\tfrac{\eta}{2};\mathbf{t})
\Bigl(\prod_{\stackrel{j=2}{j\not=i}}^M
B_w(-\epsilon_jx_j-\tfrac{\eta}{2};\mathbf{t})\Bigr)\\
&\quad\times\Bigl(
\prod_{\stackrel{j=2}{j\not=i}}^MD_w(\epsilon_jx_j-\tfrac{\eta}{2};\mathbf{t})
\Bigr)D_w(\pm x_i-\tfrac{\eta}{2};\mathbf t)
\end{split}
\end{equation*}
with the second sum running over $\bm \epsilon =(\epsilon_2,\ldots,\epsilon_{i-1},
\epsilon_{i+1},\ldots,\epsilon_M)\in\{\pm\}^{M-2}$ and
with coefficient
\begin{equation*}
\begin{split}
k_i(\bm \epsilon)&=b(\eta)\Bigl(
\sum_{\epsilon_1,\epsilon_i\in\{\pm\}}\epsilon_1\epsilon_i
\frac{b(\xi-\epsilon_1x_1-\tfrac{\eta}{2})b(\xi-\epsilon_ix_i-
\tfrac{\eta}{2})}{b(\epsilon_1x_1+\epsilon_ix_i)}\Bigr)\\
&\times\Bigl(\prod_{\stackrel{j=2}{j\not=i}}^M\epsilon_jb(\xi-\epsilon_jx_j-
\tfrac{\eta}{2})\frac{b(\epsilon_jx_j\pm x_i+\eta)}
{b(\epsilon_jx_j\pm x_i)}\Bigr)
\prod_{\stackrel{2\leq j<j^\prime\leq M}{j\not=i\not=j^\prime}}
\frac{b(\epsilon_jx_j+\epsilon_{j^\prime}x_{j^\prime}+\eta)}
{b(\epsilon_jx_j+\epsilon_{j^\prime}x_{j^\prime})}.
\end{split}
\end{equation*}
By the reflection equation (see \eqref{coeffcond3})
we conclude that $k_i(\bm \epsilon)=0$.
Hence $\overline{\mathcal{B}}_w^{(M)}(\mathbf{x};\mathbf{t})$ $=$ $L_w^{(M)}$, as desired.
\end{proof}

\subsection{The action of the $B$-operators on the pseudo vacuum vector $\Omega$}
Write for $w\in S_N$,
\begin{equation*}
R_{0 \, w(1)}(x-t_{w(1)})R_{0 \, w(2)}(x-t_{w(2)})
\cdots R_{0 \, w(N-1)}(x-t_{w(N-1)})=
\begin{pmatrix} \widehat{A}_w(x;\mathbf{t}) \! &
\! \widehat{B}_w(x;\mathbf{t})\\
\widehat{C}_w(x;\mathbf{t}) \! & \! \widehat{D}_w(x;\mathbf{t})
\end{pmatrix}
\end{equation*}
as an operator on $\C^2\otimes V$. Here the coefficients $\widehat{A}_w(x;\mathbf{t}),\ldots,
\widehat{D}_w(x;\mathbf{t})$ are operators on $V$ which act trivially
on the $w(N)$-th tensor leg of $V$
and are independent of $t_{w(N)}$.
Observe furthermore that
\begin{equation*}
T_w(x;\mathbf{t})=
\begin{pmatrix}
\widehat{A}_w(x;\mathbf{t}) &
\widehat{B}_w(x;\mathbf{t})\\
\widehat{C}_w(x;\mathbf{t}) & \widehat{D}_w(x;\mathbf{t})
\end{pmatrix}
\begin{pmatrix} A_{w(N)}(x;t_{w(N)}) &
B_{w(N)}(x;t_{w(N)})\\
C_{w(N)}(x;t_{w(N)}) & D_{w(N)}(x;t_{w(N)})
\end{pmatrix}
\end{equation*}
with $A_{w(N)}(x;t),\ldots,D_{w(N)}(x;t)$ the operators
\eqref{B1} applied to the $w(N)$-th tensor leg of $V$.
The following lemma follows from \cite[Lem. 2.3]{R}.
\begin{lem}\label{BOmega}
For $M\in\Z_{\geq 1}$, $w\in S_N$
and $\mathbf{x}=(x_1,\ldots,x_M)$,
\[
B_w^{(M)}(\mathbf{x};\mathbf{t})\Omega=
\sum_{J\subseteq\{1,\ldots,M\}}Y_w^{J}(\mathbf{x};\mathbf{t})
\Bigl(\prod_{j\in J^c}B_{w(N)}(x_j;t_{w(N)})\Bigr)
\Bigl(\prod_{i\in J}\widehat{B}_w(x_i;\mathbf{t})\Bigr)\Omega,
\]
with
\begin{equation}\label{Y}
Y_w^{J}(\mathbf{x};\mathbf{t})
:=\Bigl(\prod_{i\in J}\frac{b(x_i-t_{w(N)})}{b(x_i-t_{w(N)}+\eta)}\Bigr)
\prod_{(i,j)\in J\times J^c}\frac{b(x_i-x_j+\eta)}{b(x_i-x_j)}
\end{equation}
and $J^c:=\{1,\ldots,M\}\setminus J$ (by convention we regard the empty set as a subset of $\{1,\ldots,M\}$).
\end{lem}
Lemma \ref{BOmega} plays a crucial role in the construction of solutions of the Frenkel-Reshetikhin-Smirnov quantum KZ equations in \cite{R}.
The following formula will play a similar role for the boundary quantum KZ equations \eqref{ReflqKZ}.
\begin{cor}\label{mathcalYcor}
For $M\in\Z_{\geq 1}$ and $w\in S_N$ we have
\begin{equation}\label{key}
\begin{split}
{\overline{\mathcal{B}}}^{\xi,(M)}_w(\mathbf{x};\mathbf{t})\Omega &=
\sum_{\bm \epsilon\in\{\pm\}^M}\sum_{J\subseteq\{1,\ldots,M\}}
\mathcal{Y}^{\xi,\bm \epsilon,J}_w(\mathbf{x};\mathbf{t})
\\
& \quad \times \Bigl(\prod_{j\in J^c}B_{w(N)}(-\epsilon_jx_j-\tfrac{\eta}{2};
t_{w(N)})\Bigr) \Bigl(\prod_{i\in J} \widehat{B}_w(-\epsilon_i x_i
-\tfrac{\eta}{2};\mathbf{t})\Bigr)\Omega
\end{split}
\end{equation}
with
\begin{equation}\label{mathcalY}
\begin{split}
\mathcal{Y}^{\xi,\bm\epsilon,J}_w(\mathbf{x};\mathbf{t})&:=
\Bigl(\prod_{i=1}^M\epsilon_i b(\xi-\epsilon_i x_i-\tfrac{\eta}{2})
\prod_{r=1}^N\frac{b(\epsilon_ix_i-t_r-\tfrac{\eta}{2})}
{b(\epsilon_i x_i-t_r+\tfrac{\eta}{2})}\Bigr)\\
& \;  \times \Bigl(\prod_{1\leq i<j\leq M}\frac{b(\epsilon_i x_i+\epsilon_j x_j+\eta)}
{b(\epsilon_i x_i+\epsilon_j x_j)}\Bigr)
Y_w^J((-\epsilon_1 x_1-\tfrac{\eta}{2},\ldots,-\epsilon_M x_M-\tfrac{\eta}{2});
\mathbf{t}).
\end{split}
\end{equation}
\end{cor}
\begin{proof}
Combine Lemma \ref{BOmega} with Proposition \ref{straightening} and \eqref{eigenvalues}.
\end{proof}

\section{Proof of Theorem \ref{mr}}\label{proofSection}
Let $s_r\in S_N$ be the simple neighbouring transposition $r\leftrightarrow r+1$ ($1\leq r<N$). For $1\leq r\leq N$ write
\[
e_r \mathbf{t}:=(t_1,\ldots,t_{r-1},-t_r,t_{r+1},\ldots,t_N).
\]
\begin{lem}\label{fundrel}
For all $w\in S_N$ and $1\leq r<N$,
\begin{equation*}
\begin{split}
R_{w(r+1) \, w(r)}(t_{w(r+1)}-t_{w(r)}) T_w(x;\mathbf{t})&=
T_{w s_r}(x;\mathbf{t}) R_{w(r+1) \, w(r)}(t_{w(r+1)}-t_{w(r)}),\\
R_{w(r+1) \, w(r)}(t_{w(r+1)}-t_{w(r)}) \mathcal{U}^\xi_w(x;\mathbf{t})&=
\mathcal{U}^\xi_{w s_r}(x;\mathbf{t}) R_{w(r+1) \, w(r)}(t_{w(r+1)}-t_{w(r)}),\\
K^\xi_{w(1)}(t_{w(1)})\mathcal{U}^\xi_w(x;\mathbf{t})&=
\mathcal{U}^\xi_w(x;e_{w(1)} \mathbf{t}) K^\xi_{w(1)}(t_{w(1)}).
\end{split}
\end{equation*}
\end{lem}
\begin{proof}
The first equality follows from the Yang-Baxter equation \eqref{YBE}, the second follows immediately from the first, and the third equality follows from the reflection equation \eqref{reaux}.
\end{proof}

Fix $\xi_+,\xi_-\in\C$.
The boundary quantum KZ equations \eqref{ReflqKZ}  are equivalent to
\begin{equation}\label{ReflqKZALT}
K^{\xi_+}_r\bigl(t_r+\tfrac{\tau}{2}\bigr)F_r(\mathbf{t};\xi_-)
f(\mathbf{t})=
G_r(\mathbf{t})f(\mathbf{t}+\tau \mathbf{e}_r),\qquad r=1,\ldots,N
\end{equation}
with the linear operators $F_r(\mathbf{t};\xi_-)$
and $G_r(\mathbf{t})$  on $V$ given by
\begin{equation*}
\begin{split}
F_r(\mathbf{t};\xi_-):=&R_{Nr}(t_N+t_r)\cdots
R_{r\!+\!1 \, r}(t_{r+1}+t_r)\\
\times& R_{r\!-\!1 \, r}(t_{r-1}+t_r)\cdots R_{1r}(t_1+t_r)
K^{\xi_-}_r(t_r)\\
\times& R_{r1}(t_r-t_1)\cdots R_{r\, r\!-\!1}(t_r-t_{r-1})
\end{split}
\end{equation*}
and
\begin{equation*}
G_r(\mathbf{t}):=R_{Nr}(t_N-t_r-\tau)\cdots
R_{r\!+\!1 \, r}(t_{r+1}-t_r-\tau).
\end{equation*}

\begin{cor}\label{cor1}
For $1\leq r\leq N$ we have
\begin{equation*}
\begin{split}
F_r(\mathbf{t};\xi_-)\mathcal{U}^{\xi_-}(x;\mathbf{t})&=
\mathcal{U}^{\xi_-}_{s_r\cdots s_{N-1}}(x;e_r\mathbf{t})
F_r(\mathbf{t};\xi_-),\\
G_r(\mathbf{t})\mathcal{U}^{\xi_-}(x;\mathbf{t}+\tau \mathbf{e}_r)
&=\mathcal{U}^{\xi_-}_{s_r \cdots s_{N-1}}(x;\mathbf{t}+\tau \mathbf{e}_r)
G_r(\mathbf{t}).
\end{split}
\end{equation*}
\end{cor}
\begin{proof}
This follows by repeated application of Lemma \ref{fundrel}.
\end{proof}
\begin{cor}\label{qKZalt}
For $1\leq r\leq N$, $M\in\Z_{\geq 1}$ and $\mathbf{x}=
(x_1,\ldots,x_M)$,
\begin{equation*}
\begin{split}
F_r(\mathbf{t};\xi_-)\overline{\mathcal{B}}^{\xi_-,(M)}(\mathbf{x};
\mathbf{t})&= \Bigl(\prod_{i=1}^M\frac{b(x_i+t_r+\tfrac{\eta}{2})b(x_i-t_r-\tfrac{\eta}{2})}
{b(x_i+t_r-\tfrac{\eta}{2})b(x_i-t_r+\tfrac{\eta}{2})}\Bigr) \\
& \qquad \qquad \times \overline{\mathcal{B}}^{\xi_-,(M)}_{s_r\cdots s_{N-1}}(\mathbf{x};e_r\mathbf{t})
F_r(\mathbf{t};\xi_-),\\
G_r(\mathbf{t})\overline{\mathcal{B}}^{\xi_-,(M)}(\mathbf{x};
\mathbf{t}+\tau \mathbf{e}_r)&=
\overline{\mathcal{B}}^{\xi_-,(M)}_{s_r\cdots s_{N-1}}(\mathbf{x};
\mathbf{t}+\tau \mathbf{e}_r)
G_r(\mathbf{t}).
\end{split}
\end{equation*}
\end{cor}
\begin{proof}
Set $\mathcal{B}^{\xi,(M)}_w(\mathbf{x};\mathbf{t}):=
\prod_{i=1}^M\mathcal{B}^\xi_w(x_i;\mathbf{t})$. Since
$F_r(\mathbf{t};\xi_-)$ and $G_r(\mathbf{t})$ are acting trivially
on the auxiliary space, Corollary \ref{cor1} implies
\begin{equation*}
\begin{split}
F_r(\mathbf{t};\xi_-)\mathcal{B}^{\xi_-,(M)}(\mathbf{x};
\mathbf{t})&=\mathcal{B}^{\xi_-,(M)}_{s_r\cdots s_{N-1}}(\mathbf{x};
e_r\mathbf{t})
F_r(\mathbf{t};\xi_-),\\
G_r(\mathbf{t})\mathcal{B}^{\xi_-,(M)}(\mathbf{x};
\mathbf{t}+\tau \mathbf{e}_r)&=
\mathcal{B}^{\xi_-,(M)}_{s_r\cdots s_{N-1}}(\mathbf{x};\mathbf{t}+\tau \mathbf{e}_r)
G_r(\mathbf{t}).
\end{split}
\end{equation*}
Rewriting these relations in terms of
$\overline{\mathcal{B}}^{\xi_-}_w$ gives the desired result.
\end{proof}
Let $\mathcal{C}^{(M)}\subset\C^M$ be a discrete subset and
$w^{(r)}(\mathbf{x};\mathbf{t};\xi_+,\xi_-)$ ($\mathbf{x}\in\mathcal{C}^{(M)}$)
a weight function depending meromorphically on
$\mathbf{t}\in\C^N$.
Provided mero-uniform convergence,
\[
f_M(\mathbf{t}):=\sum_{\mathbf{x}\in\mathcal{C}^{(M)}}w^{(M)}
(\mathbf{x};\mathbf{t};\xi_+,
\xi_-)\overline{\mathcal{B}}^{\xi_-,(M)}(\mathbf{x};\mathbf{t})\Omega
\]
solves the boundary quantum KZ equations \eqref{ReflqKZALT}
iff
\begin{equation*}
\begin{split}
\sum_{\mathbf{x}\in\mathcal{C}^{(M)}}w^{(M)}(\mathbf{x};\mathbf{t};&\xi_+,\xi_-)
\Bigl(\prod_{i=1}^M\frac{b(x_i+t_r+\tfrac{\eta}{2})b(x_i-t_r-\tfrac{\eta}{2})}
{b(x_i+t_r-\tfrac{\eta}{2})b(x_i-t_r+\tfrac{\eta}{2})}\Bigr)\\
&\qquad\qquad\qquad\qquad
\times K^{\xi_+}_r(t_r+\tfrac{\tau}{2})
\overline{\mathcal{B}}^{\xi_-,(M)}_{s_r\cdots s_{N-1}}(\mathbf{x};
e_r\mathbf{t})\Omega\\
&=\sum_{\mathbf{x}\in\mathcal{C}^{(M)}}w^{(M)}(\mathbf{x};\mathbf{t}+\tau \mathbf{e}_r;
\xi_+,\xi_-)\overline{\mathcal{B}}^{\xi_-,(M)}_{s_r\cdots s_{N-1}}
(\mathbf{x};\mathbf{t}+\tau \mathbf{e}_r)\Omega
\end{split}
\end{equation*}
for $r=1,\ldots,N$ in view of Corollary \ref{qKZalt}.
Substitute \eqref{key} and use that
\[
\widehat{B}_{s_rs_{r+1}\cdots s_{N-1}}(x;e_r\mathbf{t})=
\widehat{B}_{s_rs_{r+1}\cdots s_{N-1}}(x;\mathbf{t})=
\widehat{B}_{s_rs_{r+1}\cdots s_{N-1}}(x;\mathbf{t}+\tau \mathbf{e}_r)
\]
since $s_rs_{r+1}\cdots s_{N-1}(N)=r$ and since
$\widehat{B}_w(x;\mathbf{t})$ is independent of $t_{w(N)}$.
It leads to the following result.
\begin{lem}\label{aaa}
Provided mero-uniform convergence,
\[
f_M(\mathbf{t}):=\sum_{\mathbf{x}\in\mathcal{C}^{(M)}}
w^{(M)}(\mathbf{x};\mathbf{t};\xi_+,\xi_-)
\overline{\mathcal{B}}^{\xi_-,(M)}(\mathbf{x};\mathbf{t})\Omega
\]
satisfies the boundary quantum KZ equations \eqref{ReflqKZALT} iff
\begin{equation}\label{lhs}
\begin{split}
&\sum_{\mathbf{x},\bm\epsilon,J}w^{(M)}(\mathbf{x};\mathbf{t};\xi_+,\xi_-)
\Bigl(\prod_{i=1}^M\frac{b(x_i+t_r+\tfrac{\eta}{2})b(x_i-t_r-\tfrac{\eta}{2})}
{b(x_i+t_r-\tfrac{\eta}{2})b(x_i-t_r+\tfrac{\eta}{2})}
\Bigr)
\mathcal{Y}_{s_r\cdots s_{N-1}}^{\xi_-,\bm\epsilon,J}(\mathbf{x};e_r\mathbf{t})\\
&\quad\times K^{\xi_+}_r(t_r+\tfrac{\tau}{2})
\Bigl(\prod_{j\in J^c}B_{r}(-\epsilon_jx_j-\tfrac{\eta}{2};-t_r)\Bigr)
\Bigl(\prod_{i\in J}\widehat{B}_{s_r\cdots s_{N-1}}(-\epsilon_ix_i-\tfrac{\eta}{2};
\mathbf{t})\Bigr)\Omega
\end{split}
\end{equation}
equals
\begin{equation}\label{rhs}
\begin{split}
\sum_{\mathbf{x},\bm\epsilon,J}&w^{(M)}(\mathbf{x};\mathbf{t}+\tau \mathbf{e}_r;\xi_+,\xi_-) \mathcal{Y}_{s_r\cdots s_{N-1}}^{\xi_-,\bm\epsilon,J}(\mathbf{x};\mathbf{t}+\tau \mathbf{e}_r)\\
&\times\Bigl(\prod_{j\in J^c}B_{r}(-\epsilon_jx_j-\tfrac{\eta}{2};t_r+
\tau)\Bigr)\Bigl(\prod_{i\in J}\widehat{B}_{s_r\cdots s_{N-1}}
(-\epsilon_ix_i-\tfrac{\eta}{2};\mathbf{t})\Bigr)\Omega
\end{split}
\end{equation}
for $r=1,\ldots,N$,
where the summations are over $\mathbf{x}\in\mathcal{C}^{(M)}$,
$\bm \epsilon\in\{\pm\}^M$ and over subsets $J\subseteq\{1,\ldots,M\}$.
\end{lem}
\begin{rema}\label{bbb}
In the current spin-$\frac{1}{2}$ setup we have
$B_r(x;t)B_r(y;t^\prime)\equiv 0$ as linear operators on $V$, cf. \eqref{B1}.
Hence the sum over subsets $J$ of $\{1,\ldots,M\}$
in both \eqref{lhs} and \eqref{rhs} only gives
a nonzero contribution if $\#J=M$ or $\#J=M-1$.
\end{rema}

\subsection{Highest weight contribution}
We consider in this subsection the contributions to \eqref{lhs} and \eqref{rhs}
for $J=\{1,\ldots,M\}$, which we denote by
$\mathcal{L}^{hw}_r(\mathbf{t};\xi_+,\xi_-)$ and
$\mathcal{R}^{hw}_r(\mathbf{t};\xi_+,\xi_-)$
respectively.
Since $K^{\xi_+}_r(x)\Omega=\Omega$
and $K^{\xi_+}_r(x)$ commutes with
$\widehat{B}_{s_r\cdots s_{N-1}}(y;\mathbf{t})$, we obtain by straightforward
computations,
\begin{equation*}
\begin{split}
\mathcal{L}^{hw}_r(\mathbf{t};\xi_+,\xi_-)&=\sum_{\mathbf{x},\bm \epsilon}
w^{(M)}(\mathbf{x};\mathbf{t};\xi_+,\xi_-)\\
&\times\Bigl(\prod_{i=1}^M\Bigl(\epsilon_ib(\xi_--\epsilon_ix_i-
\tfrac{\eta}{2})\prod_{\stackrel{s=1}{s\not=r}}^N\frac{b(\epsilon_ix_i-t_s-
\tfrac{\eta}{2})}{b(\epsilon_ix_i-t_s+\tfrac{\eta}{2})}\Bigr)\Bigr)\\
&\times
\Bigl(\prod_{1\leq i<j\leq M}\frac{b(\epsilon_ix_i+\epsilon_jx_j+\eta)}
{b(\epsilon_ix_i+\epsilon_jx_j)}\Bigr)
\Bigl(\prod_{i=1}^M\widehat{B}_{s_r\cdots s_{N-1}}(-\epsilon_ix_i-\tfrac{\eta}{2};
\mathbf{t})\Bigr)\Omega
\end{split}
\end{equation*}
and
\begin{equation*}
\begin{split}
\mathcal{R}^{hw}_r(\mathbf{t};\xi_+,\xi_-)&=\sum_{\mathbf{x},\bm\epsilon}
w^{(M)}(\mathbf{x};\mathbf{t}+\tau \mathbf{e}_r;\xi_+,\xi_-)
\Bigl(\prod_{i=1}^M\frac{b(\pm x_i-t_r-\tau-\tfrac{\eta}{2})}
{b(\pm x_i-t_r-\tau+\tfrac{\eta}{2})}\Bigr)\\
&\times\Bigl(\prod_{i=1}^M\Bigl(\epsilon_ib(\xi_--\epsilon_ix_i-
\tfrac{\eta}{2})\prod_{\stackrel{s=1}{s\not=r}}^N
\frac{b(\epsilon_ix_i-t_s-\tfrac{\eta}{2})}
{b(\epsilon_ix_i-t_s+\tfrac{\eta}{2})}\Bigr)\Bigr)\\
&\times
\Bigl(\prod_{1\leq i<j\leq M}\frac{b(\epsilon_ix_i+\epsilon_jx_j+\eta)}
{b(\epsilon_ix_i+\epsilon_jx_j)}\Bigr)
\Bigl(\prod_{i=1}^M\widehat{B}_{s_r\cdots s_{N-1}}(-\epsilon_ix_i-\tfrac{\eta}{2};
\mathbf{t})\Bigr)\Omega.
\end{split}
\end{equation*}

\begin{lem}\label{hwpart}
Suppose that
\[
w^{(M)}(\mathbf{x};\mathbf{t};\xi_+,\xi_-)=
\Bigl(\prod_{r=1}^N\prod_{i=1}^MF(t_r\pm x_i)\Bigr)
G_{\xi_+,\xi_-}(\mathbf{x})
\]
with $G_{\xi_+,\xi_-}(\mathbf{x})$ independent of $\mathbf{t}$. If
\[
F(x+\tau)=
\frac{b(x+\tau-\tfrac{\eta}{2})}
{b(x+\tau+\tfrac{\eta}{2})}F(x)
\]
then, provided mero-uniform convergence,
\begin{equation}\label{hwequal}
\mathcal{L}_r^{hw}(\mathbf{t};\xi_+,\xi_-)=
\mathcal{R}_r^{hw}(\mathbf{t};\xi_+,\xi_-)
\end{equation}
for $r=1,\ldots,N$.
\end{lem}
\begin{proof}
The summands of $\mathcal{L}_r^{hw}(\mathbf{t};\xi_+,\xi_-)$ and \mbox{$\mathcal{R}_r^{hw}(\mathbf{t};\xi_+,\xi_-)$} are the same except for the first line, which is also the only part depending on $t_r$.
Hence \eqref{hwequal} is valid if for all $\mathbf{x}\in\mathcal{C}^{(M)}$ and all $\bm\epsilon\in\{\pm\}^M$,
\begin{equation}\label{stept}
w^{(M)}(\mathbf{x};\mathbf{t}+\tau \mathbf{e}_r;\xi_+,\xi_-) = \Bigl(\prod_{i=1}^M\frac{b(\pm x_i-t_r-\tau+\tfrac{\eta}{2})}{b(\pm x_i-t_r-\tau-\tfrac{\eta}{2})}\Bigr)w^{(M)}(\mathbf{x};\mathbf{t};\xi_+,\xi_-).
\end{equation}
The result now follows by a direct computation.
\end{proof}

\subsection{The remaining term}
We assume that
\[
f_M(\mathbf{t})=\sum_{\mathbf{x}\in\mathbf{x}_0+\tau\Z^M}
w^{(M)}(\mathbf{x};\mathbf{t};\xi_+,\xi_-)
\overline{\mathcal{B}}^{\xi_-,(M)}(\mathbf{x};\mathbf{t})\Omega
\]
with the sum converging mero-uniformly in $\mathbf{t}\in\C^N$
and with weight function of the form
\begin{equation}\label{form}
w^{(M)}(\mathbf{x};\mathbf{t};\xi_+,\xi_-)=
\Bigl(\prod_{r=1}^N\prod_{i=1}^MF(t_r\pm x_i)\Bigr)
G_{\xi_+,\xi_-}(\mathbf{x})
\end{equation}
with $G_{\xi_+,\xi_-}(\mathbf{x})$ independent of $\mathbf{t}$ and with $F$
satisfying
\begin{equation}\label{form2}
F(x+\tau)=
\frac{b(x+\tau-\tfrac{\eta}{2})}
{b(x+\tau+\tfrac{\eta}{2})}F(x).
\end{equation}
Since the $\xi_\pm$ are fixed throughout this subsection, we will
suppress $\xi_\pm$ from the notations; in particular, we
write $w^{(M)}(\mathbf{x};\mathbf{t})$ for
$w^{(M)}(\mathbf{x};\mathbf{t};\xi_+,\xi_-)$.

By Lemma \ref{hwpart}, Lemma \ref{aaa} and Remark \ref{bbb},
$f_M(\mathbf{t})$ satisfies the boundary quantum KZ equations
\eqref{ReflqKZ} if \eqref{lhs} equals \eqref{rhs} with the sum
over $J$ running over subsets of $\{1,\ldots,M\}$ of cardinality
$M-1$. This is satisfied if for $1\leq j\leq M$ and $1\leq r\leq N$,
\begin{equation}\label{lhs1}
\begin{split}
I_{j,r}(t_r):=\frac{b(\xi_+-t_r-\tfrac{\tau}{2})}
{b(\xi_++t_r+\tfrac{\tau}{2})}&
\sum_{x_j\in x_{0,j}+\tau\Z}\sum_{\epsilon_j\in\{\pm\}}
\frac{w^{(M)}(\mathbf{x};\mathbf{t})}{b(-\epsilon_jx_j+t_r+
\tfrac{\eta}{2})}\\
\times
\Bigl(\prod_{i=1}^M&\frac{b(x_i+t_r+\tfrac{\eta}{2})b(x_i-t_r-\tfrac{\eta}{2})}
{b(x_i+t_r-\tfrac{\eta}{2})b(x_i-t_r+\tfrac{\eta}{2})}\Bigr)
\mathcal{Y}_{s_r\cdots s_{N-1}}^{\xi_-,\bm \epsilon,
\{1,\ldots,M\}\setminus\{j\}}(\mathbf{x};e_r\mathbf{t})
\end{split}
\end{equation}
equals
\begin{equation}\label{rhs1}
J_{j,r}(t_r)
:=\sum_{x_j\in x_{0,j}+\tau\Z}\sum_{\epsilon_j\in\{\pm\}}
\frac{w^{(M)}(\mathbf{x};\mathbf{t}+\tau \mathbf{e}_r)}
{b(-\epsilon_jx_j-t_r-\tau+\tfrac{\eta}{2})}
\mathcal{Y}_{s_r\cdots s_{N-1}}^{\xi_-,\bm{\epsilon},\{1,\ldots,M\}\setminus\{j\}}
(\mathbf{x};\mathbf{t}+\tau \mathbf{e}_r)
\end{equation}
for fixed
\[
\bm \epsilon_{\widehat{\jmath}}:=
(\epsilon_1,\ldots,\epsilon_{j-1},\epsilon_{j+1},\ldots,\epsilon_M)\in
\{\pm\}^{M-1},
\]
fixed $\mathbf{t}_{\widehat{r}}:=(t_1,\ldots,t_{r-1},t_{r+1},\ldots,t_N)$ (generic) and fixed
\[
\mathbf{x}_{\widehat{\jmath}}\in
(x_{0,1},\ldots,x_{0,j-1},x_{0,j+1},\dots,x_{0,M})+
\tau\Z^{M-1},
\]
where we write $\mathbf{x}_{\widehat{\jmath}}:=(x_1,\ldots,x_{j-1},x_{j+1},\ldots,x_M)$.
Note that this relies on the specific form of the $K$-matrix, in particular we have used the fact that
\[
K^\xi(x)\mathbf{e}_-=\frac{b(\xi-x)}{b(\xi+x)}\mathbf{e}_-.
\]
The expressions $I_{1,r}(t_r)$ and $J_{1,r}(t_r)$
for $M=1$ should be compared with
the left and right hand side of \eqref{special000} respectively.

We fix now $r\in\{1,\ldots,N\}$ and $j\in\{1,\ldots,M\}$, as well as
$\bm \epsilon_{\widehat{\jmath}}$, generic $\mathbf{t}_{\widehat{r}}$ and
$\mathbf{x}_{\widehat{\jmath}}$. When writing $w^{(M)}(\mathbf{x};\mathbf{t})$
we view it as function of $(x_j,t_r)$, with the remaining coordinates
being $\mathbf{x}_{\widehat{\jmath}}$ and $\mathbf{t}_{\widehat{r}}$.

Substituting \eqref{mathcalY} and \eqref{Y} and using \eqref{stept}
we get after straightforward simplifications of the summands that
\begin{equation*}
\begin{split}
I_{j,r}(t_r)&=
\frac{b(\xi_+-t_r-\tfrac{\tau}{2})}{b(\xi_++t_r+\tfrac{\tau}{2})}
\sum_{\epsilon_j\in\{\pm\}}
\sum_{x_j\in x_{0,j}+\tau\Z}
\frac{\epsilon_j m^{\epsilon_j}_{j,r}(\mathbf{x})}
{b(t_r-\epsilon_jx_j-\tfrac{\eta}{2})}
w^{(M)}(\mathbf{x};\mathbf{t}),\\
J_{j,r}(t_r)
&=\sum_{\epsilon_j\in\{\pm\}}
\sum_{x_j\in x_{0,j}+\tau\Z}
\frac{\epsilon_jm^{\epsilon_j}_{j,r}(\mathbf{x})}
{b(-t_r-\tau-\epsilon_jx_j-\tfrac{\eta}{2})}
w^{(M)}(\mathbf{x};\mathbf{t})
\end{split}
\end{equation*}
with the $t_r$-independent factor
\begin{equation*}
\begin{split}
m^{\epsilon_j}_{j,r}(\mathbf{x}):=&
b(\xi_--\epsilon_jx_j-\tfrac{\eta}{2}) \Bigl(\prod_{\stackrel{s=1}{s\not=r}}^N\frac{b(t_s-\epsilon_jx_j+\tfrac{\eta}{2})}
{b(t_s-\epsilon_jx_j-\tfrac{\eta}{2})}\Bigr) \\
\times&\Bigl(\prod_{\stackrel{i=1}{i\not=j}}^{M}\frac{b(\epsilon_jx_j\pm x_i+\eta)}
{b(\epsilon_jx_j\pm x_i)}\Bigr)\Bigl(
\prod_{\stackrel{1\leq i<i^\prime\leq M}{i\not=j\not=i^\prime}}
\frac{b(\epsilon_ix_i+\epsilon_{i^\prime}x_{i^\prime}+\eta)}
{b(\epsilon_ix_i+\epsilon_{i^\prime}x_{i^\prime})}
\Bigr)\\
\times&\prod_{\stackrel{i=1}{i\not=j}}^M
\Bigl(\epsilon_ib(\xi_--\epsilon_ix_i-\tfrac{\eta}{2})
\prod_{\stackrel{s=1}{s\not=r}}^N\frac{b(t_s-\epsilon_ix_i+\tfrac{\eta}{2})}
{b(t_s-\epsilon_ix_i-\tfrac{\eta}{2})}\Bigr).
\end{split}
\end{equation*}

We now resolve the summation over $\epsilon_j\in\{\pm\}$ for both
$I_{j,r}(t_r)$ and $J_{j,r}(t_r)$ using the boundary crossing symmetry identity \eqref{coeffcond4}.
For this we need the full ansatz on the weight function $w^{(M)}(\mathbf{x};\mathbf{t})$ as stated in Theorem
\ref{mr}; so the weight function $w^{(M)}(\mathbf{x};\mathbf{t})$ is assumed to be of the form \eqref{form} with $F$ satisfying \eqref{form2} and with
\[
G_{\xi_+,\xi_-}(\mathbf{x}):=\Bigl(\prod_{i=1}^Mg_{\xi_+,\xi_-}(x_i)\Bigr)
\prod_{1\leq i<i^\prime\leq M}h(x_i\pm x_{i^\prime})
\]
where $g_{\xi_+,\xi_-}$ and $h$ are satisfying the functional equations as stated in Theorem \ref{mr}. It implies that $G_{\xi_+,\xi_-}(\mathbf{x})$ satisfies
\begin{equation}\label{frG}
\begin{split}
G_{\xi_+,\xi_-}(\mathbf{x}-\tau \mathbf{e}_j)&=
\frac{b(\xi_-+x_j-\tfrac{\eta}{2})b(\xi_++x_j-\tfrac{\tau}{2}-
\tfrac{\eta}{2})}{b(\xi_--x_j+\tau-\tfrac{\eta}{2})
b(\xi_+-x_j+\tfrac{\tau}{2}-\tfrac{\eta}{2})}\\
&\qquad\times\Bigl(\prod_{\stackrel{i=1}{i\not=j}}^M
\frac{b(x_j\pm x_i-\tau)b(x_j\pm x_i-\eta)}
{b(x_j\pm x_i-\tau+\eta)b(x_j\pm x_i)}\Bigr)G_{\xi_+,\xi_-}(\mathbf{x}).
\end{split}
\end{equation}

Write
\begin{equation*}
U(x;t):=\frac{b(\xi_++x-\tfrac{\tau}{2}-\tfrac{\eta}{2})
b(t+x+\tfrac{\eta}{2})}{b(\xi_+-x+\tfrac{\tau}{2}-\tfrac{\eta}{2})
b(t-x+\tau+\tfrac{\eta}{2})}.
\end{equation*}
Then
\begin{equation}\label{plustominus}
\begin{split}
\sum_{x_j}
\frac{m^+_{j,r}(\mathbf{x})}
{b(t_r-x_j-\tfrac{\eta}{2})}
w^{(M)}(\mathbf{x};\mathbf{t})&=
\sum_{x_j}\frac{m^-_{j,r}(\mathbf{x})U(x_j;t_r)}
{b(t_r+x_j-\tfrac{\eta}{2})}w^{(M)}(\mathbf{x};\mathbf{t}),\\
\sum_{x_j}
\frac{m^+_{j,r}(\mathbf{x})}{b(-t_r-\tau-x_j-\tfrac{\eta}{2})}
w^{(M)}(\mathbf{x};\mathbf{t})&=
\sum_{x_j}
\frac{m^-_{j,r}(\mathbf{x})U(x_j;-t_r-\tau)}
{b(-t_r-\tau+x_j-\tfrac{\eta}{2})}w^{(M)}(\mathbf{x};\mathbf{t})
\end{split}
\end{equation}
with the sums over $x_j\in x_{0,j}+\tau\Z$
by replacing the summation variable $x_j$ in the left hand sides of
\eqref{plustominus} by $x_j-\tau$ and using the properties
\eqref{form}, \eqref{form2} and \eqref{frG}
for the weight function $w^{(M)}(\mathbf{x};\mathbf{t})$.
It follows from \eqref{plustominus} that
\begin{equation*}
\begin{split}
\sum_{\epsilon_j}\sum_{x_j}
\frac{\epsilon_jm^{\epsilon_j}_{j,r}(\mathbf{x})}
{b(t_r-\epsilon_jx_j-\tfrac{\eta}{2})}
w^{(M)}(\mathbf{x};\mathbf{t})&=
\sum_{x_j}\frac{m^-_{j,r}(\mathbf{x})V(x_j;t_r)}
{b(t_r+x_j-\tfrac{\eta}{2})}w^{(M)}(\mathbf{x};\mathbf{t}),\\
\sum_{\epsilon_j}\sum_{x_j}
\frac{\epsilon_jm^{\epsilon_j}_{j,r}(\mathbf{x})}
{b(-t_r-\tau-\epsilon_jx_j-\tfrac{\eta}{2})}
w^{(M)}(\mathbf{x};\mathbf{t})
&=
\sum_{x_j}
\frac{m^-_{j,r}(\mathbf{x})V(x_j;-t_r-\tau)}
{b(-t_r-\tau+x_j-\tfrac{\eta}{2})}w^{(M)}(\mathbf{x};\mathbf{t})
\end{split}
\end{equation*}
with
\begin{equation}\label{there}
V(x;t):=U(x;t)-1=\frac{b(\xi_++t+\tfrac{\tau}{2})b(2x-\tau)}
{b(\xi_+-x+\tfrac{\tau}{2}-\tfrac{\eta}{2})
b(t-x+\tau+\tfrac{\eta}{2})},
\end{equation}
where we have used the boundary crossing symmetry identity
\eqref{coeffcond4} for the last equality of \eqref{there}.
Since
\begin{equation*}
\begin{split}
\frac{b(\xi_+-t-\tfrac{\tau}{2})}{b(\xi_++t+\tfrac{\tau}{2})}
\frac{V(x;t)}{b(t+x-\tfrac{\eta}{2})}&=
\frac{b(\xi_+-t-\tfrac{\tau}{2})b(2x-\tau)}
{b(\xi_+-x+\tfrac{\tau}{2}-\tfrac{\eta}{2})
b(t+x-\tfrac{\eta}{2})b(t-x+\tau+\tfrac{\eta}{2})}\\
&=\frac{V(x;-t-\tau)}{b(-t-\tau+x-\tfrac{\eta}{2})}
\end{split}
\end{equation*}
it follows that
\begin{equation*}
\begin{split}
I_{j,r}(t_r)&=
\frac{b(\xi_+-t_r-\tfrac{\tau}{2})}
{b(\xi_++t_r+\tfrac{\tau}{2})}\sum_{x_j}
\frac{m^-_{j,r}(\mathbf{x})V(x_j;t_r)}
{b(t_r+x_j-\tfrac{\eta}{2})}w^{(M)}(\mathbf{x};\mathbf{t})\\
&=\sum_{x_j}\frac{m^-_{j,r}(\mathbf{x})V(x_j;-t_r-\tau)}
{b(-t_r-\tau+x_j-\tfrac{\eta}{2})}w^{(M)}(\mathbf{x};\mathbf{t})\\
&= J_{j,r}(t_r).
\end{split}
\end{equation*}
This completes the proof of Theorem \ref{mr}.

\section{Conclusions}

\noindent
Here we outline some open problems and some work in progress.\\

In this paper we obtained solutions to boundary qKZ equations in terms of bilateral series.
The domain of convergence depends on the boundary parameters $\xi_\pm$.
These formulae are expected to transform into integral solutions which are similar to the $\tau=0$ case (see Section \ref{solKZ}). 
This is a work in progress.\\

In this paper we studied the boundary qKZ equations for the $R$-matrices corresponding to $2$-dimensional irreducible representations of $U_q(\widehat{\mathfrak{sl}}_2)$ (the six-vertex model in statistical mechanics) with diagonal $K$-matrices. 
The general solution to the reflection equation for the six-vertex model is not diagonal. 
The Bethe ansatz for non-diagonal $K$-matrices involves a ''gauge transformation'' similar to the eight-vertex model, see \cite{FK} and references from this paper. 
We expect that for the non-diagonal case the formulae for solutions will have a similar structure as here, but the off-shell Bethe vectors will be more involved.\\

Integral formulae for solutions to the KZ equations can be derived as matrix elements of vertex operators realized in a Heisenberg algebra, see for example \cite{FeigFrenk,EFK}. 
In a similar way one can obtain integral formulae for solutions to the boundary qKZ equations \cite{JKKKM,JKKMW,W}.
It involves bosonization and boundary states.
Our solutions can be obtained in a similar way and such derivation  also lead to solutions expressed as integrals (not Jackson integrals) of Bethe vectors (for details of such presentations when $q=1$ see \cite{SV}).
This is a work in progress.\\

When the step size $\tau$ is a rational multiple of $\eta$ and the boundary parameters $\xi_\pm$ assume special values the boundary qKZ equations has a special class of ''rational'' solutions.
In the limit $\eta\to 0$ they correspond to solutions
of the boundary KZ equation with rational singularities at branching points. A
very special case of such solutions which is polynomial was found in
\cite{DFZJ}. Solutions describing correlation functions in the six-vertex model
with reflecting boundary conditions \cite{JKKKM} are also from this ''rational'' class.\\

The boundary qKZ equation for elliptic $R$-matrices describes correlation
functions in the eight-vertex model and in other solvable models of
statistical mechanics with elliptic parametrization of Boltzmann weights.
It is natural to expect that integral formulae for solutions will involve
a version of Bethe vectors constructed in \cite{FaTa} and in \cite{FK}.\\

The formulae for solutions to the boundary qKZ equations obtained in this paper
generalize to the higher spin representations of $U_q(\mathfrak{sl}_2)$ and
to highest weight representations of this algebra, see \cite{RSVhighspin}.\\

Solutions to the qKZ equation with $R$-matrices of $U_q(\widehat{\mathfrak{sl}}_n)$-type were
constructed, using off-shell Bethe vectors in \cite{TV}. However for other
quantum affine Lie algebras this has not been done yet. Such a project should
be relatively straightforward for classical Lie algebras where Bethe vectors
were constructed in \cite{R1,R2}. It would be also nice to do it for
dynamical $R$-matrices (SOS systems of Boltzmann weights in the language of
statistical mechanics). \\

The limit $\tau\to 0$ is similar to the WKB-type limit for KZ and qKZ equations of type A, see \cite{RV,TV}. 
In this limit the leading contribution to the solution comes from critical points of the integrand and the solution has the asymptotics $\exp(\frac{S}{\tau})(\psi+\mathcal{O}(\tau))$ as $\tau\rightarrow 0$.
Here $\psi$ is an eigenvector of the boundary transfer matrix corresponding to a solution of the Bethe equations (equations for critical points of the integrand, or in other words, the Yang-Yang action when $\tau\to 0$), and $S$ is the Yang-Yang action evaluated at this solution. 
A paper with a detailed discussion of this subject is in progress.\\

\section{Appendix. Solutions of the classical boundary KZ equations} \label{solKZ}

In this appendix we take the limit $\eta, \tau, \xi_{\pm} \to 0$ of the boundary qKZ equations
in such a way that $k:=\tau/\eta$ and $\delta_{\pm}:=\xi_{\pm}/\eta$ are constant. We obtain
classical boundary KZ equations. We construct integral solutions of the classical boundary KZ equations.

\subsection{The boundary KZ equations}

Define linear operators $r(x)$ and $\kappa^\delta(x)$ acting on $\C^2 \otimes \C^2$ and $\C^2$, respectively as the first order terms in the asymptotics of the $R$-and $K$-matrix as $\eta\rightarrow 0$,
\[ \qquad R(x;\eta) = 1 + \eta r(x) + \mathcal{O}(\eta^2) \quad \text{and} \quad K^{\eta \delta}(x) = \sigma^z + \eta \kappa^\delta(x) + \mathcal{O}(\eta^2), \quad \eta \to 0, \]
where $ \sigma^z = \left( \begin{smallmatrix} 1 & 0 \\ 0 & -1 \end{smallmatrix} \right)$.
The matrices $r(x)$ and $\kappa^\delta(x)$ have the interpretation of the classical counterparts of $R(x;\eta)$ and $K^\xi(x)$, respectively; it follows that
\[
r(x) =  \begin{pmatrix} 0 & 0 & 0 & 0 \\ 0 & -\frac{b'(x)}{b(x)} & \frac{b'(0)}{b(x)} & 0 \\ 0 & \frac{b'(0)}{b(x)} & -\frac{b'(x)}{b(x)} & 0 \\ 0 & 0 & 0 & 0 \end{pmatrix}
 \quad \text{and} \quad \kappa^\delta(x) = 2 \delta \frac{b'(x)}{b(x)} \check \sigma, \quad \text{where } \check \sigma = \begin{pmatrix} 0 & 0 \\ 0 & 1 \end{pmatrix}.
\]

It is easy to see that the transport operator $A_r(\mathbf{t};\xi_+,\xi_-;\tau)$ (see \eqref{Atauj})
featuring in the boundary qKZ equations \eqref{ReflqKZ} has the asymptotic expansion
\[ A_r(\mathbf{t};\eta \delta_+,\eta \delta_-;\eta k) = 1+ \eta a_r^\Delta(\mathbf t) + \mathcal{O}(\eta^2), \qquad \eta \to 0, \]
where $ \Delta = \delta_+ + \delta_-$ and $a_r^\Delta(\mathbf{t})$ (independent of $k$)
is given by
\begin{gather}
\label{adefinition} a_r^\Delta(\mathbf t) = -\kappa_r^\Delta(t_r)+
\sum_{s \ne r} \Theta_{rs}(t_r,t_s), \\[3mm]
\label{Thetadefn} \begin{aligned} \Theta(x,y) &=  r(x-y) + (\sigma^z \otimes 1)  r(x+y) (\sigma^z \otimes 1) \\
&= \frac{2b'(x)}{b(x \pm y)} \begin{pmatrix} 0 & 0 & 0 & 0 \\ 0 & -b(x) & b(y) & 0 \\ 0 & b(y) & -b(x) & 0 \\ 0 & 0 & 0 & 0 \end{pmatrix}. \end{aligned}
 \end{gather}
Note that $\Theta(x,y)$ is $P$-symmetric and that $\partial_y\Theta(x,y)$ is symmetric in $x$ and $y$,
since
\begin{equation*}
\partial_y\Theta(x,y)=\frac{2b^\prime(x)b^\prime(y)}{b(x\pm y)^2}
\left(\begin{matrix} 0 & 0 & 0 & 0\\
0 & -2b(x)b(y) & b(x)^2+b(y)^2 & 0\\
0 & b(x)^2+b(y)^2 & -2b(x)b(y) & 0\\
0 & 0 & 0 & 0\end{matrix}\right).
\end{equation*}
This follows from \eqref{fundamentalb} and its direct consequence
\begin{equation} \label{bderivative}
b(x)b'(y) + b'(x)b(y) = b'(0)b(x+y).
\end{equation}

In the formal limit $\eta\rightarrow 0$ the boundary quantum KZ equations \eqref{ReflqKZ}
turn into the \emph{(classical) boundary KZ equations}
\begin{equation} \label{KZ}
\qquad \qquad k\bigl(\partial_{t_r} f_\mathrm{cl}\bigr)(\mathbf{t}) = a_r^\Delta(\mathbf t) f_\mathrm{cl}(\mathbf t), \qquad r=1,\ldots,N
\end{equation}
for $\bigl(\mathbb{C}^2\bigr)^{\otimes N}$-valued local smooth functions $f_{cl}$ in $N$ variables,
where $\partial_t:=\frac{\partial}{\partial t}$. The equations \eqref{KZ} are the KZ equations for the Wess-Zumino-Witten conformal
field theory with conformal boundary conditions and boundary operator at $t=0$, cf., e.g., \cite{Car}.

The integrability of the boundary quantum KZ equations implies that \eqref{KZ} for
$k\in\mathbb{C}^\times$ define a flat connection, i.e.
\begin{equation}\label{flat}
\begin{split}
&\lbrack a_r^\Delta(\mathbf{t}), a_s^\Delta(\mathbf{t})\rbrack=0,\\
&\bigl(\partial_{t_r}a_s^\Delta\bigr)(\mathbf{t})=\bigl(\partial_{t_s}a_r^\Delta\bigr)(\mathbf{t})
\end{split}
\end{equation}
for $1\leq r,s\leq N$. In fact, the first line of \eqref{flat} follows from comparing the $\mathcal{O}(\eta^2)$-terms of the consistency equations \eqref{consistencyA} of the transport matrices for
$(\xi_+,\xi_-,\tau)=(\eta\delta_+,\eta\delta_-,0)$ as $\eta\rightarrow 0$, while the
second line of \eqref{flat} is equivalent to the symmetry of
$\partial_y\Theta(x,y)$ in $x$ and $y$. The flatness of \eqref{KZ} can also be checked
using classical Yang-Baxter and reflection equations.

\subsection{Integral solutions to the boundary KZ equations}

We formulate the main result of this appendix in this subsection, which is a description of explicit
integral solutions of the classical boundary KZ equation \eqref{KZ}. These are the formal
classical limits of the bilateral series solutions of the associated quantum KZ equations (see
Theorem \ref{mr}).

To simplify the presentation we take $b(x)$ to be $b(x)=x$ (case XXX) or $b(x)=\sinh(x)$ (XXZ),
and we will fix branch cut $\mathcal{L}_+\subset \mathbb{C}$ for the associated multi-valued function $b(x)^c$ ($c\in\mathbb{C}\setminus\mathbb{Z}$)
to be
\begin{equation*}
\mathcal{L}_+:=
\begin{cases}
\mathbb{R}_{\leq 0} & \textup{ if } b(x)=x,\\
\bigl(\mathbb{R}_{\leq 0}+2\mathbb{Z}\pi\sqrt{-1}\bigr)\cup\bigl(\mathbb{R}_{\geq 0}+
(2\mathbb{Z}+1)\pi\sqrt{-1})\qquad & \textup{ if }b(x)=\sinh(x).
\end{cases}
\end{equation*}
In the remainder of the appendix we take $b(x)^c$ to be the univalued function
$e^{\textup{Log}(b(x))c}$ on
$x\in\mathbb{C}\setminus\mathcal{L}_+$, with $\textup{Log}$ the principal logarithm.
Write
\begin{equation*}
\mathcal{L}:=
\begin{cases}
\mathbb{R} & \textup{ if } b(x)=x,\\
\mathbb{R}+\pi\sqrt{-1}\mathbb{Z} \qquad & \textup{ if } b(x)=\sinh(x)
\end{cases}
\end{equation*}
for the extension of the cut $\mathcal{L}_+$.
The results described below easily extend to the rescaled versions of $b(x)$ and to other choices of branch cuts.

To formulate the main result of the appendix we need to construct the classical analogues of the off-shell Bethe vectors
$\mathcal{B}^{\xi_-,(M)}(\mathbf{x};\mathbf{t})\Omega$ first.
Define the linear operator
\begin{equation*}
\beta(x;\mathbf{t})=2b^\prime(x)\sum_{r=1}^N
\frac{b(t_r)}{b(x\pm t_r)}\sigma_r^-,\qquad \sigma^-:=\begin{pmatrix} 0 & 0 \\ 1 & 0 \end{pmatrix}
\end{equation*}
on $\bigl(\mathbb{C}^2\bigr)^{\otimes N}$. Since $\lbrack\beta(x;\mathbf{t}),
\beta(y;\mathbf{t})\rbrack=0$ we can write for $M\geq 1$,
\[
\beta^{(M)}(\mathbf{x};\mathbf{t}):=\prod_{i=1}^M\beta(x_i;\mathbf{t}).
\]
Since $\bigl(\sigma^-\bigr)^2=0$ we have the explicit expression
\begin{equation}
\beta^{(M)}(\mathbf x;\mathbf t) = 2^M  \sum_{w \in S_M} \sum_{\mathbf{m}\in \mathcal{P}^M_N} \Biggl( \prod_{i=1}^M \frac{b^\prime(x_{w(i)})b(t_{m_i})}
{b(x_{w(i)}\pm t_{m_i})} \Biggr) \sigma^-_{\mathbf{m}}
\label{betahigherrank}
\end{equation}
where $\mathbf{m}=\{m_i\}_{i=1}^M$ runs through the set $\mathcal{P}^M_N$
of subsets of $\{1,\ldots,N\}$ of cardinality $M$ and
$\sigma^-_{\mathbf{m}} = \prod_{i=1}^M \sigma^-_{m_i}$. The $\beta^{(M)}(\mathbf{x};\mathbf{t})\Omega$ are the classical analogues of the off-shell Bethe vectors, in the sense
that
\[\overline{\mathcal{B}}^{\eta\delta_-,(M)}(\mathbf{x};\mathbf{t})=
\eta^M\Bigl(\prod_{i=1}^Mb(-x_i)\Bigr)
\beta^{(M)}(\mathbf{x};\mathbf{t})+\mathcal{O}(\eta^{M+1}),\qquad
\eta\rightarrow 0.
\]

Define
\[
g_{\Delta,\mathrm{cl}}(x):=b(x)^{\frac{2}{k}(1-\Delta)-1},\qquad
h_{\mathrm{cl}}(x):=b(x)^{\frac{2}{k}},\qquad F_{\mathrm{cl}}(x):=b(x)^{-\frac{1}{k}}.
\]
They satisfy the differential equations
\begin{equation}\label{classicaldiffeqs}
k\frac{g_{\Delta,\mathrm{cl}}^\prime(x)}{g_{\Delta,\mathrm{cl}}(x)}=(2(1-\Delta)-k)\frac{b^\prime(x)}{b(x)},
\quad
k\frac{h_{\mathrm{cl}}^\prime(x)}{h_{\mathrm{cl}}(x)}=2\frac{b^\prime(x)}{b(x)},\quad
k\frac{F_{\mathrm{cl}}^\prime(x)}{F_{\mathrm{cl}}(x)}=-\frac{b^\prime(x)}{b(x)},
\end{equation}
which are the formal classical analogues as $\eta\rightarrow 0$ of the difference equations with step size $\tau=\eta k$
satisfied by the functions $g_{\eta\delta_+,\eta\delta_-}$, $h$ and $F$ occurring as factors in the weight function of the bilateral series solutions of the boundary quantum KZ equation (see Theorem \ref{mr}). Set
\begin{equation*}
\begin{split}
w_{\mathrm{cl}}^{(M)}(\mathbf{x};\mathbf{t};\Delta)&:=G_{\Delta,\mathrm{cl}}(\mathbf{x})\prod_{i=1}^M\prod_{r=1}^N
F_{\mathrm{cl}}(t_r\pm x_i),\\
G_{\Delta,\mathrm{cl}}(\mathbf{x})&:=\Bigl(\prod_{i=1}^Mb(x_i)g_{\Delta,\mathrm{cl}}(x_i)\Bigr)
\prod_{1\leq i<j\leq M}h_{\mathrm{cl}}(x_i\pm x_j).
\end{split}
\end{equation*}

\begin{thm} \label{KZsolution}
Let $k\in\mathbb{C}^\times$ and $\Delta\in\mathbb{C}$ with $\Re(k)\leq 0$ and
$\Re((1-\Delta-N)/k)<\frac{1}{2}$. Fix $\mathbf{x}_0\in\mathbb{C}^M$ such that
$x_{0,i},x_{0,i}\pm x_{j,0}\notin\mathcal{L}$ for $1\leq i\not= j\leq M$. Then
\[
f_{M,\mathrm{cl}}(\mathbf{t}):=\int_{\mathbf{x}_0+\mathbb{R}^M}w_{\mathrm{cl}}^{(M)}(\mathbf{x};
\mathbf{t};\Delta)\beta^{(M)}(\mathbf{x};\mathbf{t})\Omega\,d^M\mathbf{x}
\]
defines an analytic $\bigl(\mathbb{C}^{2}\bigr)^{\otimes N}$-valued function on the
domain
\begin{equation}\label{domainKZ}
\{\mathbf{t}\in\mathbb{C}^N \,\, | \,\,
t_r\notin\pm x_{0,i}+\mathcal{L}\quad \forall\, r,i \}
\end{equation}
satisfying the boundary KZ equations \eqref{KZ}.
\end{thm}

It is easy to check that $f_{M,\mathrm{cl}}(\mathbf{t})$ is well-defined and analytic on the domain
\eqref{domainKZ} using the explicit expression \eqref{betahigherrank} for the off-shell Bethe vector
(note that $|\beta(x+s;\mathbf{t})|=\mathcal{O}(s^{-2})$ as $s\rightarrow\pm\infty$
if $b(x)=x$, while $|\beta(x+s;\mathbf{t})|=\mathcal{O}(s^{-1})$ as $s\rightarrow\pm\infty$
if $b(x)=\sinh(x)$). Formally $f_{M,\mathrm{cl}}(\mathbf{t})$ is obtained from the bilateral series solution
$f_M(\mathbf{t})$ of the boundary quantum KZ equations (see Theorem \ref{mr})
by taking the limit $\eta\rightarrow 0$.

\subsection{Proof of Theorem \ref{KZsolution}}

It suffices to show that $f_{M,\mathrm{cl}}(\mathbf{t})$ satisfies \eqref{KZ}.
The proof for general $M$ deviates from the proof in the quantum case because the classical
analogue of Corollary \ref{qKZalt}  is not available. Instead we will use the fact that an
explicit expression of the off-shell Bethe vector $\beta^{(M)}(\mathbf{x};\mathbf{t})\Omega$
as an explicit linear combination  of the $\sigma_{\mathbf{m}}^-$ ($\mathbf{m}\in\mathcal{P}^M_N$) is available, cf. \eqref{betahigherrank}.
We divide the proof in several steps.

\begin{lem} \label{wequations}
For $r=1,\ldots,N$ and $i=1,\ldots,M$ we have
\begin{align*}
k \frac{\partial_{t_r}w^{(M)}_\mathrm{cl}(\mathbf x;\mathbf t;\Delta)}
{w^{(M)}_\mathrm{cl}(\mathbf x;\mathbf t;\Delta)} &= \sum_{j=1}^M \frac{2b(t_r)b'(t_r)}{b(x_j\pm t_r)}, \\
k \frac{\partial_{x_i} w^{(M)}_\mathrm{cl}(\mathbf x;\mathbf t;\Delta)}{w^{(M)}_\mathrm{cl}(\mathbf x;\mathbf t;\Delta)} &= 2 \frac{b'(x_i)}{b(x_i)} \Biggl( 1-\Delta+2\sum_{j \ne i} \frac{b(x_i)^2}{b(x_i \pm x_j)} - \sum_{s=1}^N \frac{b(x_i)^2}{b(x_i\pm t_s)}\Biggr).
\end{align*}
\end{lem}
\begin{proof}
{}From \eqref{bderivative}
and \eqref{classicaldiffeqs} it follows
that $F_\mathrm{cl}$ and $G_{\Delta,\mathrm{cl}}$ satisfy
\begin{align*}
k \frac{\partial_t\bigl(F_\mathrm{cl}(t\pm x)\bigr)}{F_\mathrm{cl}(t \pm x)} &= \frac{2b(t)b'(t)}{b(x\pm t)},
\qquad k\frac{\partial_x\bigl(F_{\mathrm{cl}}(t\pm x)\bigr)}{F_{\mathrm{cl}}(t\pm x)} = -\frac{2b(x)b^\prime(x)}
{b(x\pm t)},\\
k \frac{\partial_{x_i} G_{\Delta,\mathrm{cl}}(\mathbf x)}{G_{\Delta,\mathrm{cl}}(\mathbf x)} &= 2 \frac{b'(x_i)}{b(x_i)} \Biggl( 1-\Delta+2\sum_{j \ne i} \frac{b(x_i)^2}{b(x_i \pm x_j)}\Biggr),
\end{align*}
from which the stated equations readily follow.
\end{proof}
For $r=1,\ldots,N$, introduce the set
\[
\mathcal{P}_M^N(r):=\{ \mathbf{m}\in\mathcal{P}_M^N\,\, | \,\, r\in\mathbf{m}\}.
\]
For $\mathbf{m}\in\mathcal{P}_M^N(r)$ we write $h(\mathbf{m},r)$ for the index $h\in\{1,\ldots,M\}$ such that
$m_h=r$. If $r$ is fixed and clear from context, we simply write $h(\mathbf{m})$.
\begin{lem} \label{kpartialtrf}
For $r=1,\ldots,N$, we have
\[
k\partial_{t_r}f_{M,\mathrm{cl}}(\mathbf{t})=
\frac{2^{M+1}b^\prime(t_r)}{b(t_r)}\sum_{\mathbf{m}\in\mathcal{P}_N^M}
\int_{\mathbf{x}_0+\mathbb{R}^M}w_{\mathrm{cl}}^{(M)}(\mathbf{x};\mathbf{t};\Delta)
\alpha_{\mathbf{m}}^{(r)}(\mathbf{x};\mathbf{t};\Delta)\sigma_{\mathbf{m}}^-\Omega\, d^M\mathbf{x}
\]
with
\[
\alpha_{\mathbf{m}}^{(r)}(\mathbf{x};\mathbf{t};\Delta):=
\sum_{w\in S_M}\Biggl(\prod_{i=1}^M\frac{b^{\prime}(x_{w(i)})b(t_{m_i})}
{b(x_{w(i)}\pm t_{m_i})}\Biggr)\Biggl(\sum_{j=1}^M\frac{b(t_r)^2}{b(x_j\pm t_r)}\Biggr)
\]
if $\mathbf{m}\in \mathcal{P}_N^M\setminus\mathcal{P}_N^M(r)$ and
\begin{equation*}
\begin{split}
\alpha_{\mathbf{m}}^{(r)}(\mathbf{x};\mathbf{t};\Delta):=
\sum_{w\in S_M}&\Biggl(\prod_{i=1}^M\frac{b^{\prime}(x_{w(i)})b(t_{m_i})}
{b(x_{w(i)}\pm t_{m_i})}\Biggr)\left\{ 1-\Delta+\sum_{j=1}^M\frac{b(t_r)^2}{b(x_j\pm t_r)}\right.\\
&\left.+2\sum_{j\not=w(h(\mathbf{m}))}\frac{b(x_{w(h(\mathbf{m}))})^2}
{b(x_{w(h(\mathbf{m}))}\pm x_j)}-
\sum_{s=1}^N\frac{b(x_{w(h(\mathbf{m}))})^2}{b(x_{w(h(\mathbf{m}))}\pm t_s)}\right\}
\end{split}
\end{equation*}
if $\mathbf{m}\in\mathcal{P}_N^M(r)$, where $h(\mathbf{m})=h(\mathbf{m},r)$.
\end{lem}
\begin{proof}
In
\begin{equation*}
\begin{split}
\partial_{t_r}f_{M,\mathrm{cl}}(\mathbf{t})&=
\int_{\mathbf{x}_0+\mathbb{R}^M}\bigl(\partial_{t_r}w_{\mathrm{cl}}^{(M)}(\mathbf{x};\mathbf{t};\Delta)\bigr)
\beta(\mathbf{x};\mathbf{t})\Omega d^M\mathbf{x}\\
&+\int_{\mathbf{x}_0+\mathbb{R}^M}w_{\mathrm{cl}}^{(M)}(\mathbf{x};\mathbf{t};\Delta)
\partial_{t_r}\beta(\mathbf{x};\mathbf{t})\Omega d^M\mathbf{x}
\end{split}
\end{equation*}
the partial derivative occurring in the integrand of the first integral
can be resolved using Lemma \ref{wequations}. To resolve the second partial derivative we use
\[
\partial_t\Bigl(\frac{b^\prime(x)b(t)}{b(x\pm t)}\Bigr)=b^\prime(t)b^\prime(x)\Bigl(\frac{b(x)^2+b(t)^2}
{b(x\pm t)^2}\Bigr)=-\partial_x\Bigl(\frac{b^\prime(t)b(x)}{b(x\pm t)}\Bigr)
\]
(which follows from \eqref{fundamentalb} and \eqref{bderivative}) to rewrite
$\partial_{t_r}\beta^{(M)}(\mathbf{x};\mathbf{t})$ as
\[-2^M\sum_{w\in S_M}\sum_{\mathbf{m}\in\mathcal{P}_N^M(r)}
\partial_{x_{w(h(\mathbf{m}))}}\left(\frac{b(x_{w(h(\mathbf{m}))})b^\prime(t_r)}
{b^\prime(x_{w(h(\mathbf{m}))})b(t_r)}\prod_{i=1}^M
\frac{b^\prime(x_{w(i)})b(t_{m_i})}{b(x_{w(i)}\pm t_{m_i})}\right)\sigma_{\mathbf{m}}^-
\]
and subsequently use partial integration and Lemma \ref{wequations}.
It results in the desired formula for $k\partial_{t_r}f_{M,\mathrm{cl}}(\mathbf{t})$.
\end{proof}

The next lemma deals with the other side of the classical boundary KZ equations.
\begin{lem} \label{abeta}
For $r=1,\ldots,N$ we have
\[
a_r^\Delta(\mathbf{t})\beta^{(M)}(\mathbf{x};\mathbf{t})\Omega=
\frac{2^{M+1}b^\prime(t_r)}{b(t_r)}\sum_{\mathbf{m}\in\mathcal{P}_N^M}
\gamma_{\mathbf{m}}^{(r)}(\mathbf{x};\mathbf{t};\Delta)\sigma_{\mathbf{m}}^-\Omega
\]
with
\[
\gamma_{\mathbf{m}}^{(r)}(\mathbf{x};\mathbf{t};\Delta):=
\sum_{w\in S_M}\Biggl(\prod_{i=1}^M\frac{b^{\prime}(x_{w(i)})b(t_{m_i})}
{b(x_{w(i)}\pm t_{m_i})}\Biggr)\Biggl(\sum_{j=1}^M\frac{b(t_r)^2}{b(x_j\pm t_r)}\Biggr)
\]
if $\mathbf{m}\in\mathcal{P}_N^M\setminus\mathcal{P}_N^M(r)$ and
\[
\gamma_{\mathbf{m}}^{(r)}(\mathbf{x};\mathbf{t};\Delta):=
-\sum_{w\in S_M}\Biggl(\prod_{i=1}^M\frac{b^{\prime}(x_{w(i)})b(t_{m_i})}
{b(x_{w(i)}\pm t_{m_i})}\Biggr)\Bigl(\Delta+\sum_{s\in\mathbf{m}^c}\frac{b(x_{w(h(\mathbf{m}))})^2}
{b(x_{w(h(\mathbf{m}))}\pm t_s)}\Biggr)
\]
if $\mathbf{m}\in\mathcal{P}_N^M(r)$,
where $h(\mathbf{m})=h(\mathbf{m},r)$.
\end{lem}
\begin{proof}
Let $s\not=r$. From \eqref{Thetadefn} it follows that
\begin{align*}
\Theta_{rs}(t_r,t_s) \sigma^-_r \Omega &= \frac{2b'(t_r)}{b(t_r \pm t_s)} \bigl(b(t_s) \sigma^-_s - b(t_r) \sigma^-_r\bigr) \Omega \\
\Theta_{rs}(t_r,t_s) \sigma^-_s \Omega &= \frac{2b'(t_r)}{b(t_r \pm t_s)} \bigl(b(t_s) \sigma^-_r - b(t_r) \sigma^-_s\bigr) \Omega,
\end{align*}
and hence by virtue of \eqref{fundamentalb} that
\begin{equation*}
\Theta_{rs}(t_r,t_s) \Bigl( \frac{b(t_r)}{b(x\pm t_r)}\sigma^-_r + \frac{b(t_s)}{b(x\pm t_s)}\sigma^-_s\Bigr) \Omega  =
\frac{2b'(t_r)\bigl(-b(x)^2 \sigma^-_r +b(t_r)b(t_s) \sigma^-_s \bigr)\Omega}{b(x\pm t_r)b(x\pm t_s)}.
\end{equation*}
Since $\Theta(x,y) \mathbf{e}_\pm \otimes \mathbf{e}_\pm = 0$ we conclude that
$\Theta_{rs}(t_r,t_s) \beta^{(M)}(\mathbf x;\mathbf t) \Omega$ equals
\begin{align*}
&2^{M+1} b'(t_r) \Biggl( \prod_{i=1}^M b'(x_i) \Biggr)\sum_{w \in S_M}
\sum_{\stackrel{\mathbf{m}\in \mathcal{P}^{M-1}_N:}{r,s\not\in\mathbf{m}}} \left\{\Biggl( \prod_{i=1}^{M-1} \frac{b(t_{m_i})}
{b(x_{w(i)}\pm t_{m_i})} \Biggr)\right.\\
&\left.\qquad\qquad\qquad\qquad\qquad\qquad\qquad\qquad\qquad\times
\frac{\bigl(b(t_r)b(t_s) \sigma^-_s-b(x_{w(M)})^2 \sigma^-_r\bigr)\sigma_{\mathbf{m}}^-\Omega}
{b(x_{w(M)}\pm t_r)b(x_{w(M)}\pm t_s)}\right\}\\
&=\frac{2^{M+1} b'(t_r)}{b(t_r)} \sum_{w \in S_M} \left\{
 \sum_{\stackrel{\mathbf{m} \in \mathcal{P}^M_N(s):}{r\not\in\mathbf{m}}}
 \Biggl( \prod_{i=1}^M \frac{b'(x_{w(i)})b(t_{m_i})}{b(x_{w(i)}\pm t_{m_i})} \Biggr)
 \frac{b(t_r)^2}{b(x_{w(h(\mathbf{m},s))}\pm t_r)}\sigma^-_{\mathbf{m}}\Omega \right. \\
& \left.\qquad\qquad\qquad\qquad\qquad
 -  \sum_{\stackrel{\mathbf{m}\in\mathcal{P}_N^M(r):}{s\not\in\mathbf{m}}}
 \Biggl( \prod_{i=1}^M \frac{b'(x_{w(i)})b(t_{m_i})}{b(x_{w(i)}\pm t_{m_i})} \Biggr)
 \frac{b(x_{w(h(\mathbf{m},r))})^2}{b(x_{w(h(\mathbf{m},r))}\pm t_s)}
 \sigma_{\mathbf{m}}^-\Omega
 \right\}.
 \end{align*}
Writing $\mathbf{m}^c:= \{1,\ldots,N\} \setminus \mathbf{m}$ for $\mathbf{m} \subseteq \{ 1,\ldots,N\}$ we obtain
\begin{equation} \label{abeta1}
\begin{aligned}
& \sum_{s \ne r} \Theta_{rs}(t_r,t_s) \beta^{(M)}(\mathbf x;\mathbf t) \Omega
= \frac{2^{M+1} b'(t_r)}{b(t_r)}   \\
&\,\,\times  \sum_{w \in S_M}\left\{
\sum_{\stackrel{\mathbf{m}\in \mathcal{P}^M_N:}
{r\not\in\mathbf{m}}}\Biggl( \prod_{i=1}^M \frac{b'(x_{w(i)})b(t_{m_i})}{b(x_{w(i)}\pm t_{m_i})} \Biggr)
\Biggl(\sum_{j=1}^M\frac{b(t_r)^2}{b(x_j\pm t_r)}\Biggr)\sigma_{\mathbf{m}}^-\Omega\right.\\
&\left.\quad\qquad-\sum_{\mathbf{m}\in\mathcal{P}_N^M(r)}
\Biggl( \prod_{i=1}^M \frac{b'(x_{w(i)})b(t_{m_i})}{b(x_{w(i)}\pm t_{m_i})} \Biggr)
\Biggl(\sum_{s\in\mathbf{m}^c}\frac{b(x_{w(h(\mathbf{m},r))})^2}{b(x_{w(h(\mathbf{m},r))}\pm t_s)}\Biggr)
\sigma_{\mathbf{m}}^-\Omega\right\}.
\end{aligned}
\end{equation}
Also, since $\check \sigma \sigma^- = \sigma^-$ and $\check\sigma_r\Omega= 0$, we readily obtain
\[-\kappa_r^\Delta(t_r)\beta^{(M)}(\mathbf x;\mathbf t) \Omega =-\frac{2^{M+1}b^\prime(t_r)}
{b(t_r)} \sum_{w \in S_M}  \sum_{\mathbf{m}\in \mathcal{P}^M_N(r) }
\Biggl( \prod_{i=1}^M \frac{b'(x_{w(i)})b(t_{m_i})}{b(x_{w(i)}\pm t_{m_i})} \Biggr)\Delta\sigma^-_{\mathbf{m}} \Omega. \]
Adding this to \eqref{abeta1} we obtain the lemma.
\end{proof}

\begin{proof}[Proof of Thm. \ref{KZsolution}]
Note that $\partial_{t_r}f_{M,\mathrm{cl}}(\mathbf{t})=a_r^\Delta(\mathbf{t})f_{M,\mathrm{cl}}(\mathbf{t})$ if and only if
\[
\int_{\mathbf{x}_0+\mathbb{R}^M}w_{\mathrm{cl}}^{(M)}(\mathbf{x};\mathbf{t};\Delta)
\alpha_{\mathbf{m}}^{(r)}(\mathbf{x};\mathbf{t};\Delta)d^M\mathbf{x}=
\int_{\mathbf{x}_0+\mathbb{R}^M}w_{\mathrm{cl}}^{(M)}(\mathbf{x};\mathbf{t};\Delta)
\gamma_{\mathbf{m}}^{(r)}(\mathbf{x};\mathbf{t};\Delta)d^M\mathbf{x}
\]
for all $\mathbf{m}\in\mathcal{P}_N^M$. It  thus suffices to show that
\[
\alpha_{\mathbf{m}}^{(r)}(\mathbf{x};\mathbf{t};\Delta)=\gamma_{\mathbf{m}}^{(r)}(\mathbf{x};\mathbf{t};\Delta)
\]
for all $\mathbf{m}\in\mathcal{P}_N^M$. This is trivial if $r\not\in\mathbf{m}$.

For the remainder of the proof fix $\mathbf{m}\in\mathcal{P}_N^M(r)$ and write $h=h(\mathbf{m},r)$.
Then Lemma \ref{kpartialtrf} and Lemma \ref{abeta} yield
\begin{equation*}
\begin{split}
W_{\mathbf{m}}^{(r)}(\mathbf{x};&\mathbf{t}):=
\frac{\alpha_{\mathbf{m}}^{(r)}(\mathbf{x};\mathbf{t};\Delta)-\gamma_{\mathbf{m}}^{(r)}(\mathbf{x};\mathbf{t};\Delta)}
{\prod_{i=1}^Mb(t_{m_i})b^\prime(x_i)}\\
&=\sum_{w\in S_M}\Biggl(\prod_{i=1}^M\frac{1}{b(x_{w(i)}\pm t_{m_i})}\Biggr)\left\{
1+2\sum_{i\not=h}\frac{b(x_{w(h)})^2}{b(x_{w(h)}\pm x_{w(i)})}\right.\\
&\left.\quad\qquad\qquad\qquad\qquad\qquad\qquad+\sum_{i=1}^M\frac{b(t_r)^2}{b(x_{w(i)}\pm t_r)}-
\sum_{i=1}^M\frac{b(x_{w(h)})^2}{b(x_{w(h)}\pm t_{m_i})}\right\}.
\end{split}
\end{equation*}
By \eqref{fundamentalb} we have $b(y)^2-b(x)^2=b(y\pm x)$, which can be used to write
\begin{equation*}
\begin{split}
W_{\mathbf{m}}^{(r)}(\mathbf x;\mathbf t) =
\sum_{w \in S_M} \Biggl( \prod_{j=1}^M  \frac{1}{b(x_{w(j)}\pm t_{m_j})}\Biggr)
&\sum_{i\not=h}\left\{\frac{b(t_r)^2}{b(x_{w(i)}\pm t_r)}\right.\\
&\left.\quad -\frac{b(x_{w(h)})^2}{b(x_{w(h)}\pm t_{m_i})}
+2\frac{b(x_{w(h)})^2}{b(x_{w(h)}\pm x_{w(i)})}\right\}.
\end{split}
\end{equation*}
Similarly we have
\begin{align*}
\frac{b(x_j)^2}{b(x_j \pm x_i)}  + \frac{b(t)^2}{b(x_i\pm t)} &= \frac{b(x_i)^2b(x_j\pm t )}{b(x_j \pm x_i)b(x_i\pm t)},\\
\frac{b(x_j)^2}{b(x_j \pm x_i)} - \frac{b(x_j)^2}{b(x_j \pm t)} &= \frac{b(x_j)^2b(x_i\pm t)}{b(x_j \pm x_i)b(x_j\pm t )},
\end{align*}
which leads to
\begin{equation*}
\begin{split}
W_{\mathbf{m}}^{(r)}(\mathbf{x};\mathbf{t})&=\sum_{i\not=h}\sum_{w\in S_M}
\Biggl(\prod_{\stackrel{j=1}{j\not=h,i}}^M\frac{1}{b(x_{w(j)}\pm t_{m_j})}\Biggr)
\frac{1}{b(x_{w(h)}\pm x_{w(i)})}\\
&\qquad\quad\times\left\{\frac{b(x_{w(i)})^2}{b(x_{w(i)}\pm t_r)b(x_{w(i)}\pm t_{m_i})}+
\frac{b(x_{w(h)})^2}{b(x_{w(h)}\pm t_r)b(x_{w(h)}\pm t_{m_i})}\right\}.
\end{split}
\end{equation*}
For fixed $i\not=h$, the $w$-term of the inner sum is cancelled by the $w\sigma_i$-term, where $\sigma_i\in S_M$ is the transposition $i\leftrightarrow h$.
Hence $W_{\mathbf{m}}^{(r)}(\mathbf{x};\mathbf{t})=0$,
as desired.
\end{proof}


\end{document}